\documentclass[10pt]{amsart}

\usepackage[T1]{fontenc}

\usepackage{float}
\usepackage[margin=1.4in]{geometry} 

\usepackage[foot]{amsaddr}

\usepackage{mathtools, amsmath, amssymb, amsthm, stmaryrd}

\usepackage{graphicx}

\usepackage{tikz-cd, tikz}
\usepackage[all]{xy}

\usepackage{amssymb, mathrsfs}

\usepackage{todonotes, verbatim}

\usepackage{enumerate}

\usepackage{booktabs}

\usepackage{aliascnt}

\usepackage[colorlinks, citecolor=forestgreen, linkcolor=frenchblue, pagebackref=true]{hyperref}
\usepackage[nameinlink, capitalize,noabbrev]{cleveref} 
\renewcommand*{\backrefalt}[4]{%
\ifcase #1 %
No citations.%
\or
(Cited on page #2).%
\else
(Cited on pages #2).%
\fi
}

\usepackage{xcolor}
\definecolor{forestgreen}{rgb}{0.13, 0.55, 0.13}
\definecolor{frenchblue}{rgb}{0.0, 0.45, 0.73}

\usepackage{csquotes}

\usepackage[normalem]{ulem}

\usepackage{bbm}

\usepackage{halloweenmath}

\makeatletter
\DeclareDocumentCommand{\newmathcommand}{ m O{0} m }{%
  \ifcsname\expandafter\@gobble\string#1\space\endcsname
    \expandafter\expandafter\expandafter\let\expandafter\csname old\string#1\expandafter\endcsname\expandafter=\csname\expandafter\@gobble\string#1\space\endcsname
  \else
    \expandafter\let\csname old\string#1\endcsname=#1
  \fi
  \expandafter\newcommand\csname new\string#1\endcsname[#2]{#3}
  \DeclareRobustCommand#1{%
    \ifmmode
      \expandafter\let\expandafter\next\csname new\string#1\endcsname
    \else
      \expandafter\let\expandafter\next\csname old\string#1\endcsname
    \fi
    \next
  }%
}
\makeatother
\newmathcommand{\r}{\mathring}

\makeatletter
\DeclareRobustCommand\bigop[2][1]{%
  \mathop{\vphantom{\sum}\mathpalette\bigop@{{#1}{#2}}}\slimits@
}
\newcommand{\bigop@}[2]{\bigop@@#1#2}
\newcommand{\bigop@@}[3]{%
  \vcenter{%
    \sbox\z@{$#1\sum$}%
    \hbox{\resizebox{\ifx#1\displaystyle#2\fi\dimexpr\ht\z@+\dp\z@}{!}{$\m@th#3$}}%
  }%
}
\makeatother
\let\originalbeef\bigstar
\renewcommand{\bigstar}{\bigop[0.8]{\originalbeef}}

\let\originalleft\left
\let\originalright\right
\renewcommand{\left}{\mathopen{}\mathclose\bgroup\originalleft}
\renewcommand{\right}{\aftergroup\egroup\originalright}

\AddToHook{cmd/@setemails/before}{\raggedright}

\AtBeginDocument{
\let\originalref\ref
\let\originaleqref\eqref
\renewcommand{\ref}[1]{\autoref{#1}}
\renewcommand{\eqref}[1]{{\renewcommand{\ref}{\originalref}\originaleqref{#1}}}
}

\renewcommand{\phi}{\varphi}
\renewcommand{\epsilon}{\varepsilon}

\allowdisplaybreaks

\addtocontents{toc}{\protect\setcounter{tocdepth}{1}}

\DeclareMathOperator{\Sub}{Sub}

\DeclareMathOperator{\Marks}{Marks}

\DeclareMathOperator{\Stab}{Stab}

\DeclareMathOperator{\Tr}{Tr}

\DeclareMathOperator{\im}{im}

\DeclareMathOperator{\Mack}{Mack}

\DeclareMathOperator{\Set}{Set}
\DeclareMathOperator{\Twin}{Twin}
\DeclareMathOperator{\Fin}
{\textup{\textsf{Fin}}}
\DeclareMathOperator{\Alg}{Alg}
\DeclareMathOperator{\Span}{\textup{\textsf{Span}}}
\DeclareMathOperator{\Bispan}{\textup{\textsf{Bispan}}}

\DeclareMathOperator{\Barks}{Barks}

\newcommand{\id}{\mathrm{id}}

\newcommand{\N}{\mathbb{N}}
\newcommand{\BA}{\mathbb{A}}
\newcommand{\Z}{\mathbb{Z}}
\newcommand{\Q}{\mathbb{Q}}

\newmathcommand{\P}{\mathcal{P}}

\newmathcommand{\1}{\mathbbm{1}}

\newcommand{\mymacro}{\Omega}

\makeatletter
\providecommand{\leftsquigarrow}{%
  \mathrel{\mathpalette\reflect@squig\relax}%
}
\newcommand{\reflect@squig}[2]{%
  \reflectbox{$\m@th#1\rightsquigarrow$}%
}
\makeatother

\numberwithin{equation}{section}

\crefname{equation}{}{}

\newaliascnt{theorem}{equation}
\newtheorem{theorem}[theorem]{Theorem}
\aliascntresetthe{theorem}

\newaliascnt{proposition}{theorem}
\newtheorem{proposition}[proposition]{Proposition}
\aliascntresetthe{proposition}

\newaliascnt{corollary}{theorem}
\newtheorem{corollary}[corollary]{Corollary}
\aliascntresetthe{corollary}

\newaliascnt{lemma}{theorem}
\newtheorem{lemma}[lemma]{Lemma}
\aliascntresetthe{lemma}

\theoremstyle{definition}

\newaliascnt{definition}{theorem}
\newtheorem{definition}[definition]{Definition}
\aliascntresetthe{definition}

\newaliascnt{remark}{theorem}
\newtheorem{remark}[remark]{Remark}
\aliascntresetthe{remark}

\newaliascnt{example}{theorem}
\newtheorem{example}[example]{Example}
\aliascntresetthe{example}

\newaliascnt{notation}{theorem}
\newtheorem{notation}[notation]{Notation}
\aliascntresetthe{notation}

\makeatletter
\def\retheorem@title{???}
\newtheoremstyle{retheorem}%
    {.5\baselineskip plus .2\baselineskip minus .2\baselineskip}
    {.5\baselineskip plus .2\baselineskip minus .2\baselineskip}
    {\normalfont}%
    {0pt}%
    {\bfseries}%
    {.}%
    {5pt plus 1pt minus 1pt}%
    {\retheorem@title}
\theoremstyle{retheorem}
\newtheorem*{@retheorem}{}

\makeatother

\title{Composition of Bispans of $G$-Sets and Plethysm}

\author{Nathan Cornelius}
\email{nathan.cornelius@uky.edu}
\author{Evan Franchere}
\email{evan.franchere@uky.edu}
\author{Usman Hafeez}
\email{usman.hafeez@uky.edu}
\author{Jesse Keyes}
\email{jdke228@uky.edu}
\author{David Mehrle}
\email{davidm@uky.edu}
\author{Lakshay Modi}
\email{lakshay.modi@uky.edu}
\author{Nathaniel Stapleton}
\email{nat.j.stapleton@uky.edu}
\address{Department of Mathematics, University of Kentucky, Lexington, KY, U.S.A.}
\date{\today}

\begin{document}

\begin{abstract}
Let $P(G)$ be the Grothendieck ring of the semiring of endomorphisms of the point in the $1$-category of bispans of finite $G$-sets for a finite group $G$. This is the bispan analogue of the Burnside ring of $G$. The ring $P(G)$ admits a third operation from composition of bispans. We produce a character map for $P(G)$ landing in a plethory built out of polynomial rings and the poset of conjugacy classes of subgroups of $G$. We prove that the character map sends composition of bispans to the plethysm operation -- which is a generalization of composition of polynomials. 
\end{abstract}

\maketitle

\vspace{-1em}
\tableofcontents

\section{Introduction}
Bispans of finite $G$-sets have played an increasingly important role in $G$-equivariant stable homotopy theory \cite{BH2015,BH2018,CHLL2024,CHLL2025,EH2023}. In $G$-equivariant algebra, the objects that play the role of commutative rings are $G$-Tambara functors: product preserving functors from the category of bispans to the category of sets \cite{Tambara1993}. The Grothendieck ring of the endomorphisms of $G/G$ in the category of bispans, which we denote $P(G)$, is the bispan analogue of the Burnside ring of $G$.  In $G$-equivariant stable homotopy theory, $P(G)$ arises as $\pi_{0}^{G}$ of the free $N_{\infty}$-algebra on a point, see \cref{HmtpyTheory}. In $G$-equivariant algebra, it arises as the top level of the free Tambara functor over the Burnside Tambara functor on one generator with trivial action. An important feature of this ring is that it admits an extra operation coming from the composition of bispans. Algebraic structures combining ring structures with composition have previously been studied under the name of plethories \cite{BorgerWieland2005}. 
One might hope that the composition on $P(G)$ has plethystic features. While $P(G)$ is not necessarily a plethory for all finite groups, the goal of this paper is to show that 
its composition product can be understood via a character map that lands in a plethory and sends the composition operation on $P(G)$ to the plethysm operation of the target. 

Let $\P^G$ be the $1$-category of bispans of finite $G$-sets. This category is built similarly to the $1$-category of spans of finite $G$-sets. For finite $G$-sets $X$ and $Y$, $\P^G(X,Y)$ is the set of isomorphism classes of bispans from $X$ to $Y$. The set $\P^G(X,Y)$ admits the structure of a semiring. Let $P(G)$ be the Grothendieck ring of the semiring $\P^G(G/G,G/G)$. We call this the Burnside--Tambara ring. This ring admits an extra operation coming from composition of bispans. When $G = e$, the trivial group, $P(e)$ is isomorphic to $\Z[x]$ and the composition of bispans corresponds to composition of polynomials.

We denote an isomorphism class of bispans from $G/G$ to $G/G$ as
\[
[X \xrightarrow{\varphi} Y],
\]
where $X$ and $Y$ are finite $G$-sets, abbreviating the isomorphism class of the bispan $[G/G \leftarrow X \xrightarrow{\varphi} Y \rightarrow G/G]$ by omitting the source and target. Addition and multiplication of isomorphism classes of bispans are given by
\begin{align*}
	[X \xrightarrow{\varphi} Y] + [X'\xrightarrow{\varphi'}Y'] 
		&= [X \amalg X' \xrightarrow{\varphi \amalg \varphi'} Y \amalg Y'] \text{ and } \\
	[X \xrightarrow{\varphi} Y] \cdot [X' \xrightarrow{\varphi'} Y'] 
		&= [(X \times Y') \amalg (Y \times X') \xrightarrow{(\varphi \times \id_{Y'}) + (\id_Y \times \varphi')} Y \times Y'], 
\end{align*}
with $0 = [\emptyset \rightarrow \emptyset]$ and $1 = [\emptyset \rightarrow \ast]$. The composition operation on $P(G)$ is quite a bit harder to describe as it involves a ``dependent product'' and ``exponential diagram'' (see \cref{CompDiagram}).

The Burnside ring of $G$, $A(G)$, sits inside of $P(G)$ as bispans of the form $[\emptyset \to Y]$, and we view $P(G)$ as an $A(G)$-algebra. The study of Burnside rings benefits from a well-behaved character theory. Let $\Marks(G)$ denote the ring of integer valued functions on the set of conjugacy classes of subgroups of $G$. 
The character map
\[
\chi \colon A(G) \to \Marks(G)
\]
sends a finite $G$-set $X$ to the function 
\[
    [H] \mapsto |X^H|
\]
on conjugacy classes of subgroups of $G$. 
Dress proved that the character map is injective and a rational isomorphism \cite{Dress}. There are extensions of this character theory for $A(G)$ to $P(G)$ that have distinct benefits. We focus on a particular extension of this character theory in which the target exhibits stellar algebraic properties.

Let $A^+(G)$ be the effective Burnside semiring -- the semiring of isomorphism classes of finite $G$-sets. For $H \subseteq G$, let $I_H \subseteq A^+(G)$ be the image of the transfer map $\Tr_{H}^{G} \colon A^+(H) \to A^+(G)$. This is the summand of $A^+(G)$ generated by basis elements of the form $[G/K]$ for $K \subseteq H$. 

Consider the commutative monoid ring $\Z[I_H]$; this is a polynomial ring with generators of the form $[G/K]$ for $K \subseteq H$. Given a conjugacy class $[H]$ of $G$, we produce a ring map
\[
\chi_{[H]} \colon P(G) \to \Z[I_H]
\]
by sending a bispan $[X \xrightarrow{\varphi} Y]$  to 
\[
\sum_{y \in Y^H} \left[\Tr_{H}^{G} (\varphi^{-1}(y))\right],
\]
where the sum takes place in $\Z[I_H]$. These maps $\chi_{[H]}$ assemble to give a character map
\[
\chi \colon P(G) \to \Barks(G),
\]
where $\Barks(G)$ is the product of the commutative monoid rings $\Z[I_H]$ over conjugacy classes of subgroups of $G$: 
\[
	\Barks(G) = \prod_{[H]} \Z[I_H].
\]

Closely related to work of Thevenaz in \cite{Thevenaz1988}, we prove:
\begin{proposition}
The map $\chi \colon P(G) \to \Barks(G)$ is a ring homomorphism. If $G$ is a Dedekind group, then it is injective and a rational isomorphism.
\end{proposition}

So far we have said little about composition, but that is the main point of this story. For a commutative ring $R$, $R$-plethories are complicated algebraic gadgets: roughly speaking, they are $R$-algebras that are equipped with a coaddition map, comultiplication map, additive and multiplicative counit maps and a coscalar map, all in the category of $R$-algebras -- as well as a composition product, called plethysm, mimicking ordinary composition of single variable polynomials.  The polynomial ring $\Z[x]$ admits the structure of a $\Z$-plethory. In this case, the structure maps described above come from the fact that it ``corepresents'' the identity functor on the category of commutative rings and the composition product is just ordinary composition of polynomials. Because of all of the structure that they possess, plethories tend to be important objects when they occur.

The ring $\Barks(G)$ is determined in a simple way by the poset of conjugacy classes of subgroups of $G$. For any finite poset $P$, let $\Marks(P) = \prod_{p \in P} \Z$, and let $\Barks(P) = \prod_{p \in P} \Z[x_q \mid q \leq p]$. Note that $\Barks(P)$ is a $\Marks(P)$-algebra. If $P$ is the poset of conjugacy classes of subgroups of $G$, then $\Marks(G) = \Marks(P)$ and $\Barks(G) \cong \Barks(P)$ as $\Marks(P)$-algebras. It turns out that $\Barks(P)$ admits the structure of a $\Marks(P)$-plethory.

We will focus attention on the composition operation (or plethysm) on $\Barks(P)$. Note that if $q \leq p$ then $\Z[x_r \mid r \leq q]$ is a subring of $\Z[x_r \mid r \leq p]$. We write $(x_r)^p$, when $r \leq p$, for the element of $\Barks(P)$ that is $x_r$ in the factor corresponding to $p$ and $0$ elsewhere. The plethysm on $\Barks(P)$ is determined by the formula
\[
(x_q)^p \circ (x_r)^u = \begin{cases}
        (x_r)^p, &\text{if $u=q$}\\
        0, &\text{otherwise.}
        \end{cases}
\]
This formula extends to arbitrary pairs of elements in $\Barks(P)$ since $(x_q)^p \circ (-)$ is a ring map for any $q \leq p$ and $(-) \circ f$ is a ring map for any $f \in \Barks(P)$. We note that this operation is quite simple -- merely combining the poset structure and the rules for composition of polynomials -- and does not, a priori, have anything to do with composition of bispans.

\begin{theorem} \label{mainthm}
The character map
\[
\chi \colon P(G) \to \Barks(G)
\]
sends the composition operation on $P(G)$ to the plethysm on $\Barks(G)$. 
\end{theorem}

Thus, for $G$ a Dedekind group, the plethory structure on $\Barks(G)$ fully captures the composition operation in $P(G)$. 

It is natural to wonder if $P(G)$ itself admits the structure of an $A(G)$-plethory. When $G = C_p$ is the cyclic group of order $p$ ($p$ prime), \cref{Cpplethory} proves that this is indeed the case. However, in \cref{C4notplethory}, we use \cref{mainthm} to show that $P(C_4)$ admits no $A(C_4)$-plethory structure compatible with the composition of bispans.

\subsection*{Acknowledgments} It is a pleasure to thank Mike Hill, Tomer Schlank, Ben Spitz, and Noah Wisdom for helpful conversations. The senior author thanks Lars Hesselholt for his suggestion to spend more time thinking about plethories. This paper came out of the Kentucky Bourbon Seminar and it is a pleasure to thank the other attendees including Josh Peterson and Will Stroupe.

\section{Burnside--Tambara Rings}

In this section we review the definition of the Burnside ring, introduce the Burnside--Tambara ring, and prove some first results regarding the Burnside--Tambara ring.

\subsection{Burnside Rings}

Let $G$ be a finite group and let $\Fin^G$ denote the category of finite left $G$-sets and $G$-equivariant maps. 
We construct a commutative semiring $A^+(G)$ as follows: as a set, $A^+(G)$ consists of the isomorphism classes of finite $G$-sets. Given a finite $G$-set $X$, we denote its isomorphism class by $[X] \in A^+(G)$. The sum and product in $A^+(G)$ are induced by the disjoint union and Cartesian product, respectively: 
\[
	[X] + [Y] = [X \amalg Y] 
	\qquad 
	\text{and}
	\qquad
	[X] \cdot [Y] = [X \times Y].
\] 
The additive and multiplicative identities are given by $0 = [\emptyset]$ and $1 = [G/G]$, respectively. As a commutative monoid under addition, $A^+(G)$ is the free commutative monoid on the set of isomorphism classes of transitive $G$-sets.

\begin{definition}\ 
\begin{enumerate}[(a)]
	\item We call the commutative semiring $A^+(G)$ the effective Burnside semiring of $G$.
	\item The Burnside ring of $G$, $A(G)$, is the Grothendieck group of $A^+(G)$ with respect to addition.
\end{enumerate}
\end{definition}

Using the canonical inclusion $A^+(G) \hookrightarrow A(G)$, we identify $A^+(G)$ with its image in the Burnside ring. In particular, we use the notation $[X] \in A(G)$ for the class of a finite $G$-set $X$ in the Burnside ring. If an element of $A(G)$ lies in the image of $A^+(G)$, we say that it is effective; it represents the isomorphism class of a finite $G$-set. Otherwise, we say that the element is virtual.

As a $\Z$-module, $A(G)$ is free with canonical basis given by the isomorphism classes of transitive $G$-sets: the classes $[G/H] \in A(G)$ as $H$ ranges over subgroups of $G$. Because two orbits $G/H$ and $G/K$ are isomorphic as $G$-sets if and only if $H$ is conjugate to $K$ in $G$, this basis is in bijection with conjugacy classes of subgroups of $G$. We can describe the product in $A(G)$ on this basis using the double-coset formula: there is an isomorphism of $G$-sets
\[
	G/H \times G/K \cong \coprod_{HgK \in H \backslash G / K} G/(^g\!H \cap K),
\]
where $^g\!H \coloneqq gHg^{-1}$. 

\begin{example}
	Let $G = C_p$ be the cyclic group of prime order $p$. There is an isomorphism of commutative rings
	\[
		A(C_p) \cong \Z[t] / (t^2 - pt),
	\]
    where $t = [C_p/e]$ and $1 = [C_p/C_p]$.  
\end{example}

\begin{example}
Let $G = C_4$ be the cyclic group of order $4$. There is an isomorphism of commutative rings
\[
A(C_4) \cong \Z[u,v]/(u^2-2u,v^2-4v,uv-2v),
\]
where $u = [C_4/C_2]$ and $v = [C_4/e]$.
\end{example}

\subsection{Spans and Burnside Rings}
\label{subsec:spans and burnside rings}

There is another description of the effective Burnside semiring $A^+(G)$, the utility of which will become clear once we generalize this construction. 

Let $\Span(\Fin^G)$ be the $1$-category of spans on $\Fin^G$. Objects of this category are finite $G$-sets. Let $X, Y\in \Span(\Fin^G)$. A morphism from $X$ to $Y$ is the isomorphism class of a span, i.e. a diagram of the form
\[
	\begin{tikzcd}
		X & A \ar[l] \ar[r] & Y
	\end{tikzcd}
\]
in $\Fin^G$. Two spans from $X$ to $Y$ are isomorphic if there is an isomorphism $A \xrightarrow{\cong} A'$ such that the following diagram commutes: 
\[
	\begin{tikzcd}[row sep=small]
		& A \ar[dl] \ar[dr] \ar[dd, "\cong" description] & \\
		X & & Y. \\
		& A' \ar[ul] \ar[ur] 
	\end{tikzcd}
\]
We write $[X \leftarrow A \rightarrow Y]$ for the isomorphism class of the span $X \leftarrow A \rightarrow Y$. Abusing terminology, we will refer to an isomorphism class of spans as a span. Composition in $\Span(\Fin^G)$ is given by pullback in $\Fin^G$, as in \cref{fig:span composition}. 
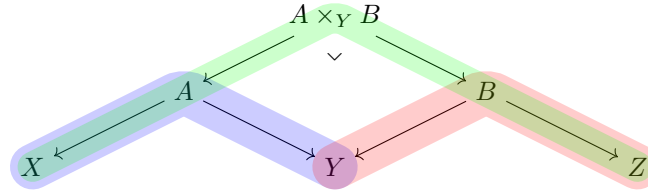
\begin{figure}[h!]
	\begin{center}
		\begin{tikzpicture}
			\node (X) at (0,0) {$X$};
			\node (A) at (2,1) {$A$};
			\node (Y) at (4,0) {$Y$};
			\node (B) at (6,1) {$B$};
			\node (Z) at (8,0) {$Z$};
			\node (C) at (4,2) {$A \times_Y B$};
			\draw[->] (A) -- (X);
			\draw[->] (A) -- (Y);
			\draw[->] (B) -- (Y);
			\draw[->] (B) -- (Z);
			\draw[->] (C) -- (A);
			\draw[->] (C) -- (B);
			\draw[blue,  line width=6mm, rounded corners, line cap=round, opacity=0.2]
				(0,0) -- (2,1) -- (4,0);
			\draw[red,   line width=6mm, rounded corners, line cap=round, opacity=0.2]
				(4,0) -- (6,1) -- (8,0);
			\draw[green, line width=4mm, rounded corners, line cap=round, opacity=0.2]
				(0,0) -- (4,2) -- (8,0);
			\node at (4,1.5) {\rotatebox{-45}{$\lrcorner$}};
		\end{tikzpicture}
	\end{center}
	\caption{Composition of spans. The span $[X \leftarrow A \rightarrow Y]$ is highlighted in blue, the span $[Y \leftarrow B \rightarrow Z]$ is highlighted in red, and the composite span $[Y \leftarrow B \rightarrow Z] \circ [X \leftarrow A \rightarrow Y]$ is highlighted in green. The central diamond of the figure is a pullback square in $\Fin^G$. }
	\label{fig:span composition}
\end{figure}

The category $\Span(\Fin^G)$ has all finite products, given by the ordinary disjoint union of $G$-sets.

\begin{proposition}
	For any two finite $G$-sets $X$ and $Y$, $\Span(\Fin^G)(X,Y)$ is a commutative monoid with addition given by disjoint union.  
\end{proposition}

In particular, $\Span(\Fin^G)(\ast,\ast)$ is not only a commutative monoid, but also a commutative semiring, with multiplication given by composition. This gives us a second description of the effective Burnside semiring. 

\begin{corollary}
	There is an isomorphism of commutative semirings
	\[
		\begin{tikzcd}[row sep=0]
			\Span(\Fin^G)(\ast,\ast) 
				\ar[r] 
				&
			A^+(G) ,
				\\
			{[\ast \leftarrow X \rightarrow \ast]}
				\ar[r, mapsto] 
				&
			{[X].}
		\end{tikzcd}
	\]
\end{corollary}

With this description, the Burnside ring $A(G)$ is the Grothendieck group of the semiring of spans from a point to a point.

\subsection{Bispans}
Bispans generalize spans by replacing the middle object of a span by a morphism. 

\begin{definition}
	A bispan from $X$ to $Y$ in the category of finite $G$-sets is a diagram in $\Fin^G$ of the form 
    \[
        \begin{tikzcd}
			X 
			& A \ar[l, "\varphi"'] \ar[r, "\psi"] 
			& B \ar[r, "\omega"] 
			& Y.
		\end{tikzcd}
    \]
	Two bispans from $X$ to $Y$ are  isomorphic if there are isomorphisms $\alpha \colon A \xrightarrow{\cong} A'$ and $\beta \colon B \xrightarrow{\cong} B'$ such that the following diagram commutes: 
	\[
		\begin{tikzcd}[row sep = small]
			& A \ar[r, "\psi"]\ar[dl,"\varphi"']\ar[dd, "\alpha"',"\cong"] & B\ar[dd, "\beta"',"\cong"] \ar[dr, "\omega"] \\
			X & & & Y \\
			& A' \ar[r, "\psi'"'] \ar[ul, "\varphi'"] & B' \ar[ur, "\omega'"']
		\end{tikzcd}
	\]
	We write $[X \xleftarrow{\varphi} A \xrightarrow{\psi} B \xrightarrow{\omega} Y]$ for the isomorphism class of the  bispan $X \xleftarrow{\varphi} A \xrightarrow{\psi} B \xrightarrow{\omega} Y$. We will abuse terminology and refer to an isomorphism class of bispans as a bispan.
\end{definition}

Bispans in $\Fin^G$ are the morphisms in a category $\P^G \coloneqq \Bispan(\Fin^G)$ whose objects are finite $G$-sets. Composition in $\P^G$ is described in \cite[Section 7]{Tambara1993}. This composition is more complicated than composition of spans; we defer its description to \cref{sec:Composition of Bispans} below.  

\begin{remark}
	Bispans are commonly called polynomials in the modern category theory literature because they arise in connection with polynomial functors \cite{GambinoKock2013}. This motivates the choice of notation $\P^G$, which is known as the category of polynomials in $\Fin^G$. However, we will use the term bispan in order to distinguish between polynomials $f \in R[x_1, \ldots, x_n]$ and polynomial morphisms $[X \leftarrow A \rightarrow B \rightarrow Y]$ in $\P^G$. 
\end{remark}

Each object $Y\in\P^G$ canonically admits the structure of a commutative semiring object. For more detail, see the discussion before \cref{ring_from_composition}. This endows each morphism set $\P^G(X,Y)$ with the structure of a commutative semiring:

\begin{proposition}[{\cite[Proposition 7.6]{Tambara1993}}]
\label{proposition:semiring structure on bispans}
	For any two finite $G$-sets $X$ and $Y$, $\P^G(X,Y)$ is a commutative semiring with operations given as follows. If $\Sigma, \Sigma' \in \P^G(X,Y)$ are the bispans 
	\[
		\Sigma = [X \xleftarrow{\varphi} A \xrightarrow{\psi} B \xrightarrow{\omega} Y] \qquad \text{ and } \qquad 
		\Sigma' = [X \xleftarrow{\varphi'} A' \xrightarrow{\psi'} B' \xrightarrow{\omega'} Y],
	\]
	then
	\begin{align*}
		\Sigma + \Sigma'  
			&= 
		[X \xleftarrow{\varphi+\varphi'} A \amalg A' \xrightarrow{\psi \amalg \psi'} B \amalg B' \xrightarrow{\omega + \omega'} Y] \\
		\Sigma \cdot \Sigma' 
			&= 
		[X \xleftarrow{\widehat{\varphi}} (A \times_Y B') \amalg (B \times_Y A') \xrightarrow{\widehat{\psi}} B \times_Y B' \xrightarrow{\widehat{\omega}} Y],
	\end{align*}
	where 
	\begin{itemize}
		\item $\widehat{\varphi} = (\varphi \circ \pi_A) + (\varphi' \circ \pi_{A'})$ is given by projecting onto $A$ or $A'$ and then applying $\varphi$ or $\varphi'$; 
		\item $\widehat{\psi} = (\psi \times \id_{B'}) + (\id_B \times \psi')$ is given by applying $\psi$ to $A$ or $\psi'$ to $A'$;  and 
		\item $\widehat{\omega} = \omega \circ \pi_B = \omega' \circ \pi_{B'}$ is given by projecting onto either factor and then applying $\omega$ or $\omega'$.  
	\end{itemize}
	The additive and multiplicative identities in this semiring are 
	\[
		0 = [X \leftarrow \emptyset \rightarrow \emptyset \rightarrow Y] \qquad \text{ and } \qquad 
		1 = [X \leftarrow \emptyset \rightarrow Y \xrightarrow{\id_Y} Y]. 
	\]
\end{proposition}

\subsection{Burnside--Tambara Rings}

The isomorphism $A^+(G) \cong \Span(\Fin^G)(\ast,\ast)$ suggests studying $\P^G(\ast,\ast)$, the commutative semiring of bispan endomorphisms of a point, and its Grothendieck group. 

\begin{notation}
	When the source and target of a bispan are both a point, we omit them from the notation, shortening $[\ast \leftarrow A \xrightarrow{\psi} B \rightarrow \ast]$ to $[A \xrightarrow{\psi} B]$ or sometimes to just $[\psi]$. 
\end{notation}

By \cref{proposition:semiring structure on bispans}, $\P^G(\ast,\ast)$ is a commutative semiring with operations
\begin{align*}
	[A \xrightarrow{\psi} B] + [A'\xrightarrow{\psi'}B'] 
		&= [A \amalg A' \xrightarrow{\psi \amalg \psi'} B \amalg B'] \text{ and } \\
	[A \xrightarrow{\psi} B] \cdot [A' \xrightarrow{\psi'} B'] 
		&= [(A \times B') \amalg (B \times A') \xrightarrow{(\psi \times \id_{B'}) + (\id_B \times \psi')} B \times B'], 
\end{align*}
with $0 = [\emptyset \rightarrow \emptyset]$ and $1 = [\emptyset \rightarrow \ast]$. We introduce the notation 
\[
	P^+(G) \coloneqq \P^G(\ast,\ast)
\] 
for this commutative semiring.

\begin{definition}\ 
\begin{enumerate}[(a)]
	\item We call the commutative semiring $P^+(G)$ the effective Burnside--Tambara semiring of $G$. 
	\item The Burnside--Tambara ring of $G$ is the ring $P(G)$ given by taking the Grothendieck group of $P^+(G)$ with respect to addition. 
\end{enumerate}
\end{definition}

We again identify elements of $P^+(G)$ with their image under the canonical inclusion into $P(G)$ and call such elements effective; they are represented by bispans. Otherwise we refer to an element as virtual.

\begin{remark}
	We name these rings Burnside--Tambara rings to emphasize their connections both to the classical Burnside rings and to Tambara functors, for which the polynomial categories $\P^G$ form the domains \cite{Tambara1993}. These rings assemble into a Tambara functor $G/H \mapsto P(H)$ isomorphic to the representable Tambara functor $\P^G(*,-)$. 

	While we are primarily interested in the rings $P(G)$ here, there is prior interest in these representable Tambara functors \cite{HMQ2023,MQS2024,MQS2025,CMQSV2025}. 
\end{remark}

\begin{remark}
\label{HmtpyTheory}
The Burnside--Tambara ring of $G$ arises in equivariant stable homotopy theory via \cite[Theorem 4.3]{Ullman}. We provide two other sketches of arguments explaining this, the first of which is closely related to Ullman's result.

Let $\text{Sp}^G_{\geq 0}$ denote the category of connective $G$-spectra and let $\mathcal{O}$ be a complete $\mathbb{N}_\infty$-operad. The functor $\underline{\pi}_0 \colon \text{Sp}^G_{\geq 0} \rightarrow \Mack_G$ is strong symmetric monoidal with right adjoint given by the Eilenberg-Mac Lane functor $H(-)$, inducing an adjunction between the categories of $\mathcal{O}$-algebras in $\text{Sp}_{\geq 0}$ and $\Mack_G$, respectively. One can verify that the square of right adjoints commutes: 

\begin{center}
\begin{tikzpicture}
\node (TL) at (0,2) {$\text{Alg}_{\mathcal{O}}(\text{Sp}^G_{\geq 0})$};
\node (BL) at (0,0) {$\text{Sp}^G_{\geq 0}$};
\node (TR) at (4,2) {$\text{Alg}_{\mathcal{O}}(\text{Mack}_G)$};
\node (BR) at (4,0) {$\text{Mack}_G$,};

\draw[->] (TR) to node[above] {$H(-)$} (TL);
\draw[->] (TL) to (BL);
\draw[->] (BR) to node[below] {$H(-)$} (BL);
\draw[->] (TR) to (BR);
\end{tikzpicture}
\end{center}
where the vertical arrows are the functors which forget $\mathcal{O}$-algebra structure and have left adjoints given by a free construction. Consequently, the square of left adjoints commutes. Evaluating both composites of the left adjoints at the equivariant sphere spectrum $\mathbb{S}_G$ yields the desired result: Recall that $\underline{\pi}_0(\mathbb{S}_G)$ is the Burnside Mackey functor. By \cite[Definition 3.3.0]{HMQ2023}, the free complete Tambara functor on the Burnside Mackey functor yields $\underline{P}$. Evaluating this Tambara functor at level $G/G$ gives the Burnside-Tambara ring for $G$.

This can also be proved Lawvere theoretically: the free $\mathcal{O}$-algebra on a point is a connective $G$-spectrum and hence may be viewed as a model of a (multi-sorted) Lawvere theory $\mathcal{L}$ in the $\infty$-category of spaces (see \cite[Section 4.2]{CHLL2024} for more details). This Lawvere theory is the group-completion of the $(2,1)$-category of bispans on the category of finite $G$-sets. The free $\mathcal{O}$-algebra on a point is given by $\mathcal{L}(*, -)$, and applying $\pi_0^G$ applies $\pi_0$ on the top level of $\mathcal{L}(*,-)$, giving us $\pi_0\mathcal{L}(*, *) \cong P(G)$.
\end{remark}

\begin{example}
    Let $G = e$ be the trivial group. The Burnside--Tambara ring of $e$ is the polynomial ring over the integers in one variable
    \[
        P(e) \cong \mathbb{Z}[x].
    \] 
    The isomorphism is given by 
     \[
    	\begin{tikzpicture}
			\node (A1) at (0.03,0) {$\big[\emptyset$};
			\node (A2) at (1,0) {$\underline{k}\big]$};
			\node (A3) at (2,0) {$k$};
			\node at (3,0) {and};
			\node (B1) at (4,0) {$\big[\ast$};
			\node (B2) at (5,0) {$\ast\big]$};
			\node (B3) at (6,0) {$x,$};
			\draw[->] (A1) -- (A2);
			\draw[|->] (A2) -- (A3);
			\draw[->] (B1) -- (B2);
			\draw[|->] (B2) -- (B3);
		\end{tikzpicture}
    \]
    where $\underline{k} = \{1,2,\ldots,k\}$.
\end{example}

\begin{example}
\label{BurnsideTambaraRingOfC2} \label{big_burnside_ring_Cp_example_computation}
    Let $C_p$ be the cyclic group of order $p$. Then a presentation of the Burnside--Tambara ring of $C_p$ is
    \[
        P(C_p) \cong A(C_p)[x,n]/(tx^p - tn),
    \]
    where $A(C_p) \cong \mathbb{Z}[t] / (t^2 - pt)$ is the Burnside ring of $C_p$, and $t = [C_p/e]$ is the class of a free orbit. 
    The isomorphism is determined by 
    \[
    	\begin{tikzpicture}[xscale=1.5]
			\node (A1) at (0.03,2) {$\big[\emptyset$};
			\node (A2) at (0.815,2) {$C_p/e\big]$};
			\node (A3) at (2, 2) {$t,$};
			\node (B1) at (0,1) {$\big[\ast$};
			\node (B2) at (1,1) {$\ast\big]$};
			\node (B3) at (2,1) {$x,$};
			\node (C1) at (0.15,0) {$\big[C_p/e$};
			\node (C2) at (1,0) {$\ast\big]$};
			\node (C3) at (2,0) {$n.$};
			\draw[->] (A1) -- (A2);
			\draw[|->] (A2) -- (A3);
			\draw[->] (B1) -- (B2);
			\draw[|->] (B2) -- (B3);
			\draw[->] (C1) -- (C2);
			\draw[|->] (C2) -- (C3);
		\end{tikzpicture}
    \]
    In general, for any $C_p$-set $X$, $[\emptyset \to X] \longmapsto [X] \in A(C_p)$. The case $p=2$ was considered in \cite[Lemma 3.6]{BH2019}. 
    \end{example}

    \begin{example}\label{big_burnside_ring_C4_example_computation}
    	Let $C_4$ be the cyclic group of order 4. The Burnside--Tambara ring of $C_4$ is 
	\[	
		P(C_4) \cong A(C_4)[x,m,n,s]/(x^4v - nv, x^2u - mu,x^2v- mv,sv-2x^2v, su-2s, sm - sx^2, s^2-2un),
	\]
	where $A(C_4) \cong \Z[u,v]/(u^2 - 2u, v^2 - 4v, uv - 2v)$ is the Burnside ring of $C_4$, and $u = [C_4/C_2]$ and $v = [C_4/e]$.  
	The isomorphism is determined by 
    \[
    	\begin{tikzpicture}[xscale=1.5]
			\node (A1) at (0,2) {$\big[\emptyset$};
			\node (A2) at (1,2) {$C_4/C_2\big]$};
			\node (A3) at (2,2) {$u,$};
			\node (B1) at (0,1) {$\big[\emptyset$};
			\node (B2) at (1,1) {$C_4/e\big]$};
			\node (B3) at (2,1) {$v,$};
			\node (C1) at (0,0) {$\big[\ast$};
			\node (C2) at (1,0) {$\ast\big]$};
			\node (C3) at (2,0) {$x,$};
			
			\node (D1) at (4,2) {$\big[C_4/C_2$};
			\node (D2) at (5,2) {$\ast\big]$};
			\node (D3) at (6,2) {$m,$}; 
			\node (E1) at (4,1) {$\big[C_4/e$};
			\node (E2) at (5,1) {$\ast\big]$};
			\node (E3) at (6,1) {$n,$}; 
			\node (F1) at (4,0) {$\big[C_4/e$};
			\node (F2) at (5,0) {$C_4/C_2\big]$};
			\node (F3) at (6,0) {$s.$}; 
			\draw[->]  (A1) -- (A2);
			\draw[|->] (A2) -- (A3);
			\draw[->]  (B1) -- (B2);
			\draw[|->] (B2) -- (B3);
			\draw[->]  (C1) -- (C2);
			\draw[|->] (C2) -- (C3);
			\draw[->]  (D1) -- (D2);
			\draw[|->] (D2) -- (D3);
			\draw[->]  (E1) -- (E2);
			\draw[|->] (E2) -- (E3);
			\draw[->]  (F1) -- (F2);
			\draw[|->] (F2) -- (F3);
		\end{tikzpicture}
    \]
    With these identifications, $\big[C_4/C_2 \to C_4/C_2\big] = ux$ and $\big[C_4/e \to C_4/e\big] = vx$. This presentation can be verified using \cref{character_map_rational_iso}.
    \end{example}

\subsection{Structure of Burnside--Tambara Rings}

In \cref{BurnsideTambaraRingOfC2}, we saw that bispans of the form $[\emptyset \to X] \in P(C_p)$ correspond to classes of $C_p$-sets $[X] \in A(C_p)$. Indeed, this holds for any group. 
\begin{proposition}
    The map $A(G) \to P(G)$ determined by $[X] \mapsto [\emptyset \to X]$ is an injective ring homomorphism.
\end{proposition}

This homomorphism makes $P(G)$ into an $A(G)$-algebra and, in particular, an $A(G)$-module. We note that, even in the $G = C_2$ case, $P(G)$ is not a free $A(G)$-module. Like $A(G)$, the Burnside--Tambara ring $P(G)$ has a canonical additive basis.

\begin{proposition}
\label{basis_for_burnside_tambara_ring}
As a $\Z$-module, $P(G)$ has a basis given by 
\[
    \big\{ 
        \left[X \to G/H\right]
        \ \big\vert\ 
        X \in \Fin^G \text{and } H \subseteq G
    \big\}.
\]
\end{proposition}

\begin{proof}
For a given bispan $[X \xrightarrow{\varphi} Y]$, we can decompose $Y$ into orbits as $Y \cong 
\coprod_{i=1}^n G/H_i$, and then 
\[
    [X \xrightarrow{\varphi} Y] = \sum_{i=1}^n \left[\varphi^{-1}\left(G/H_i\right) \to G/H_i\right]
\]
gives the unique decomposition.
\end{proof}

  Note that by \cref{basis_for_burnside_tambara_ring}, the basis elements for $P(G)$ are all of the form 
\begin{equation*}
    \left[ \coprod_{i=1}^{n} G/K_i \xrightarrow[]{\psi} G/L \right]. 
\end{equation*}
The next proposition shows that these can be assumed to be in a special form.

\begin{proposition}\label{basis_elements_are_canonical_quotients}
Every basis element of $P(G)$ described in \cref{basis_for_burnside_tambara_ring} is of the form
\begin{equation*}
    \left[ \coprod_{i=1}^{n} G/K_i \xrightarrow[]{\psi} G/L \right],
\end{equation*}
where each $K_i \subseteq L$ and the map $\psi_i \colon G/K_i \to G/L$ on each summand is the canonical quotient map $eK_i \mapsto eL$.
\end{proposition}
\begin{proof}
    Note that the existence of a $G$-equivariant morphism 
    $G/K_i \to G/L$ implies the existence of an element $g_i \in G$ which 
    witnesses the fact that $K_i$ is $G$-subconjugate to $L$ i.e., 
    $g_iK_ig_i^{-1} \subseteq L$. Using these $g_i$'s, we can 
    construct commutative squares 
    
    \[\begin{tikzcd}
    	{\coprod_{i=1}^{n}G/K_i} && G/L \\
    	\\
    	{\coprod_{i=1}^{n}G/g_iK_ig_i^{-1}} && G/L,
    	\arrow["{\psi_i}", from=1-1, to=1-3]
    	\arrow["{K_i \mapsto K_i g_i^{-1}}", from=1-1, to=3-1]
    	\arrow["{\text{id}}"', from=1-3, to=3-3]
    	\arrow[from=3-1, to=3-3]
    \end{tikzcd}\]
    where the two vertical legs are isomorphisms and the bottom
    map is the sum of canonical quotient maps.
\end{proof}
There is a natural grading on the commutative ring $P(G)$. 

\begin{definition}[Graded Ring Structure]
A bispan $[X \xrightarrow{\varphi} Y]$ is homogeneous of degree $m$ if $|\varphi^{-1}(y)| = m$ for all $y \in Y$.
\end{definition}

\noindent With this grading, the Burnside ring $A(G)$ is the degree $0$ part of $P(G)$. We write $P(G)_m$ for the abelian group of homogeneous degree $m$ elements of $P(G)$. 

In order to better understand the canonical basis for $P(G)$, let $\P^G_{(G/K, G/H)}(*,*)$ denote the subset of $P(G)$ consisting of bispans of the form $[* \leftarrow G/K \to G/H \to *]$. This set is in bijective correspondence with the set of isomorphism classes of arrows in the arrow category of finite $G$-sets with source $G/K$ and target $G/H$. For $H \subseteq G$, consider the inclusion $\Sub(H) \hookrightarrow \Sub(G)$ of sets of subgroups. There is a conjugation action of the normalizer $N_G(H)$ on the source and an action of $G$ by conjugation on the target. The quotient by these actions is well-defined and gives a map 
\[
	\Phi_H \colon \Sub(H)/N_{G}(H) \to  \Sub(G)/G.
\]

The purpose of the following proposition is to show that the additive basis for $P(G)$ can be quite complicated -- in contrast to $A(G)$.

\begin{proposition} \label{prop:isoclassesarrows}
    There is a canonical bijection between the set $\P^G_{(G/K, G/H)}(*,*)$ and the fiber $\Phi_H^{-1}([K])$.
\end{proposition}
\begin{proof}
	Recall that an equivariant map $G/K \to G/H$ is determined by the image of the identity coset. 
	For $g \in G$, let $\varphi_g \colon G/K \to G/H$ be the equivariant map determined by $\varphi_g(eK) = gH$. Consider the map
	\[
	\begin{tikzcd}[row sep=0]
		F \colon \P^G_{(G/K, G/H)}(*,*) \ar[r] & \Phi^{-1}_H({[K]}),\\
		{[G/K \xrightarrow[]{\varphi_g} G/H]} \ar[r, mapsto] & {[g^{-1}Kg]}.
	\end{tikzcd}
	\]

   To see that $F$ is well-defined, consider two isomorphic bispans $[G/K \xrightarrow[]{\varphi_{g_1}} G/H]$ and 
   $[G/K \xrightarrow[]{\varphi_{g_2}} G/H]$ so that the 
   following diagram 
   \[
	\begin{tikzcd}
		{G/K} \ar[r, "\varphi_{g_1}"]\ar[d,"eK \mapsto g_4K"',"\cong"] & {G/H} \ar[d, "\cong"', "eH \mapsto g_3H"] \\
		{G/K} \ar[r, "\varphi_{g_2}"] & {G/H},
	\end{tikzcd}
	\] 
    commutes where $g_3 \in N_G(H)$ and $g_4 \in N_G(K)$.

Then $F([G/K \xrightarrow[]{\varphi_{g_1}} G/H]) = [g_1^{-1} K g_1 ]$ and 
$F([G/K \xrightarrow[]{\varphi_{g_2}} G/H]) = [g_2^{-1} K g_2]$, and we want to show that $g_1^{-1}Kg_1$ and $g_2^{-1}Kg_2$ are $N_G(H)$-conjugate subgroups of $H$. 
The diagram above asserts that 
\[
	g_1g_3H = g_4g_2H. 
\]
In particular, there exists an $h \in H$ such that 
$g_1 g_3 h = g_4 g_2$. Left-multiplying by $g_4^{-1}$ shows that 
$g_4^{-1} g_1 g_3 h = g_2$. Then  
\begin{align*}
    g_2^{-1} K g_2 &= (g_4^{-1} g_1 g_3 h)^{-1} K (g_4^{-1} g_1 g_3 h)\\
    &= h^{-1} g_3^{-1} g_1^{-1} g_4 K g_4^{-1} g_1 g_3 h\\
    &= h^{-1} g_3^{-1} g_1^{-1} K g_1 g_3 h &\text{ since } g_4 \in N_G(K)\\
    &= (g_3 h)^{-1} (g_1^{-1} K g_1) (g_3 h). 
\end{align*}
This witnesses the fact that $g_1^{-1}K g_1$ and $g_2^{-1}Kg_2$ are conjugate by an element of $N_G(H)$, namely $g_3h$. Hence, $F$ is well-defined.

To see that $F$ is injective, suppose we have $[g_1^{-1}Kg_1]$ and $[g_2^{-1}Kg_2]$ that 
are $N_G(H)$-conjugate. This means that 
there exists a $g \in N_{G}(H)$ such that 
\begin{equation*}
    g_1^{-1}Kg_1 = g^{-1}g_2^{-1}Kg_2g.
\end{equation*}
and, moreover, that 
\begin{equation*}
    K = g_1g^{-1}g_2^{-1}Kg_2gg_1^{-1}.
\end{equation*}
This implies that $g_1g^{-1}g_2^{-1} \in N_G(K)$. Then the following commutative diagram
   \[
	\begin{tikzcd}
		{G/K} \ar[r, "\varphi_{g_1}"]\ar[d,"eK \mapsto g_1g^{-1}g_2^{-1}K"',"\cong"] & {G/H} \ar[d, "\cong"',"eH \mapsto g^{-1}H"] \\
		{G/K} \ar[r, "\varphi_{g_2}"] & {G/H},
	\end{tikzcd}
	\] 
    shows that the bispans $[G/K \xrightarrow[]{\varphi_{g_1}} G/H]$ and $[G/K \xrightarrow[]{\varphi_{g_2}} G/H]$ are isomorphic. This shows the injectivity. 

To see that $F$ is surjective, recall that there exists a $G$-equivariant map $G/K \to G/H$ sending $eK \to gH$ if and only if $g^{-1}Kg \subseteq H$. It follows that for each conjugate of $K$ lying inside $H$ 
there is an isomorphism class of $G$-equivariant maps of the form $[G/K \to G/H]$ that maps to it under $F$. This shows the required bijection. 
\end{proof}

If we specialize the group $G$ to be Dedekind (a group whose subgroups are all normal) things simplify:  

\begin{proposition}
\label{dedekind_grps_character_iso_lemma}
    For $G$ a Dedekind group, there is a bijection between the degree $m$ additive basis elements of $P(G)$ of the form $[X \to G/H]$ and the collection of (multi)sets of subgroups $\{ K_1, \ldots, K_n \}$ of $H$ satisfying
    \begin{equation*}
        \sum_{i=1}^n |H/K_i| = m. 
    \end{equation*}
\end{proposition}   

\begin{proof}
    Decompose $X$ into orbits 
    $X \cong \coprod_i G/K_i$, for some $K_i \subseteq G$. Let $\varphi \colon \coprod_{i} G/K_i \to G/H$ be a map of $G$-sets, and let $\{ \varphi_i \colon G/K_i \to G/H \}$
    be the set of maps out of the summands. Recall that there exists a 
    $G$-equivariant map $G/K_i \to G/H$ if and only if there exists
    an $g \in G = N_G(H)$ such that $g^{-1} K_i g = K_i \subseteq H$,
    where the equality is due to the fact that $G$ is Dedekind. 
    This yields the required set of subgroups  of $H$. The degree $m$ elements of the basis 
    are the bispans that have fibers with cardinality
    $m$. It follows that 
    \begin{equation*}
     \varphi^{-1}(eH) \cong \coprod_i \varphi_i^{-1}(eH).    
    \end{equation*}
    Taking the cardinality of both sides, we have 
    \begin{equation*}
        |\varphi^{-1}(eH)| = \sum_{i}|\varphi_i^{-1}(eH)| = \sum_i |H/K_i|. 
    \end{equation*}
    To complete the bijection, note that for any set of subgroups 
    $\{ K_i \}$ of $H$ that satisfy 
    $\sum_{i} |H/K_i| = m$, we can construct the bispan $[\coprod_i G/K_i \xrightarrow[]{\psi_e} G/H]$ which maps each $eK_i$ to $eH$. To see this map is an inverse, it suffices to show that there is a unique isomorphism class of maps $[\coprod_i G/K_i \to G/H].$ Because every subgroup of $G$ is normal, this follows from \cref{basis_elements_are_canonical_quotients}.
\end{proof}

Finally, we note that in the case of a Dedekind group there is a particularly simple finite set of generators of $P(G)$ as a commutative ring. It is known (see \cite[Proposition 3.12]{schuchardt2025algebraicallyclosedfieldsequivariant}) that $P(G)$ is a finitely generated commutative ring for any finite $G$.

\begin{proposition}
For $G$ a Dedekind group, the semiring $P^+(G)$ is generated by the finite set
\[
\left\{\left[ \left( \coprod_{K \subseteq H} \coprod_{i = 1}^{j_K} G/K \right) \to G/H \right]\ \middle\vert\ 0 \leq j_K \leq |G/H| \text{ for all } K \subseteq H \right\}.
\]
\end{proposition}
\begin{proof}
For $G$ a Dedekind group, the double coset formula implies that
\[
[G/K \to G/G] \cdot [Y \to G/H] = \left[Y \amalg \coprod_{i=1}^{|G/H|} G/K \to G/H \right]
\]
for $K \subseteq H$. Thus, products of elements in the finite set described in the proposition give rise to an additive basis for $P^+(G)$.

We see from the proof that we really only need $0 \leq j_K \leq |G/H|-1$ for $H$ a proper subgroup of $G$ in the statement of the proposition.
\end{proof}

\section{Character Theory}

In this section, we review the character theory for the Burnside ring and generalize it to the Burnside--Tambara ring. We describe the relationship between this character theory and the work of Thevenaz, as well as a further generalization involving transfer systems.
\subsection{Character Theory for Burnside Rings}

Let $\Sub(G)$ denote the set of subgroups of $G$, and write $\Sub(G)/G$ for the set of $G$-conjugacy classes of subgroups of $G$. For $H \in \Sub(G)$, we write $[H] \in \Sub(G)/G$ for the conjugacy class of $H$. Recall that $\Sub(G)/G$ is a partially ordered set, where $[K] \leq [H]$ if $K$ is subconjugate to $H$ (i.e., there exists $g \in G$ such that $K^g \subseteq H$).

\begin{definition}\ 
\begin{enumerate}[(a)]
	\item The ring of marks of $G$ is the ring of $\Z$-valued functions on the set of conjugacy classes of subgroups of $G$: 
	\[
		\Marks(G) \coloneqq \Set(\Sub(G)/G,\Z) \cong \prod_{\Sub(G)/G} \Z. 
	\]
	\item The marks homomorphism (or character map) is the ring homomorphism
	\[
		\chi \colon A(G) \to \Marks(G)
	\]
	given by sending $[X] \in A(G)$ to the function $\chi([X])$ with value at $[H]$ given by
	\[
		\chi([X])_{[H]} = |X^H|,
	\]
	where $X^H$ is the set of $H$-fixed points of the $G$-set $X$.
\end{enumerate}
\end{definition}

To see that the marks homomorphism is well-defined, it suffices to note that for a $G$-set $X$, $g \in G$, and $H \subseteq G$, the action of $g^{-1}$ gives a bijection $X^H \cong X^{(H^g)}$.

\begin{theorem}[\cite{Dress}, \cite{TomDieck79}]  \label{classicalmarks}
	The marks homomorphism $\chi \colon A(G) \to \Marks(G)$ is an injective ring homomorphism, and it is an isomorphism after rationalization (i.e.\ $\chi \otimes \Q$ is an isomorphism). 
\end{theorem}

	We may describe the image of $A(G)$ inside $\Marks(G)$ using a table of marks: a table whose entries are $\chi([G/H])_{[K]} = |(G/H)^K|$, with rows labeled by the isomorphism classes of transitive $G$-sets $[G/H]$ and columns labeled by conjugacy classes of subgroups $[K] \in \Sub(G)/G$. This table is the change-of-basis matrix from the basis of $\Q \otimes \Marks(G)$ given by indicator functions on $\Sub(G)/G$ to the basis of $\Q \otimes \Marks(G)$ given by $(\chi\otimes\Q)([G/H])$.

\begin{example}
\label{tableofmarks}
Below are the tables of marks for $C_p$ (left) and $C_4$ (right).

\begin{center}

\renewcommand{\arraystretch}{1.2}
\begin{tabular}{lcc}
         \toprule
         $ $
         &  $e$ 
         & $C_p$
         \\ \midrule
          $[C_p/e]$
         &  
         $p$
         & 
          $0$ \\
          $[C_p/C_p]$
          &
          $1$
          &
          $1$
        \\  \bottomrule
        \phantom{1}
        \\
\end{tabular} 
\hspace*{1cm}
\begin{tabular}{lccc}
         \toprule
         $ $
         &  $e$  & $C_2$ & $C_4$ \\ 
         \midrule
          $[C_4/e]$ &   $4$ &  $0$ & $0$ \\
          $[C_4/C_2]$ & $2$ & $2$ & $0$ \\
          $[C_4/C_4]$ & $1$ & $1$ & $1$
          \\
          \bottomrule
\end{tabular} 

\end{center}
\end{example}

\subsection{Character Theory for Burnside--Tambara Rings} 

Here, we develop a character theory for the Burnside--Tambara ring analogous to that of the Burnside ring. We will define the character map for an arbitrary finite group, but it will only be injective if the group is Dedekind. The target of this character theory will be a $\Marks(G)$-algebra called $\Barks(G)$.

Recall that, for $H \subseteq G$, there is an additive transfer map (also called an induction map)
\[
\Tr_{H}^{G} \colon A^+(H) \to A^+(G)
\]
defined by $\Tr_{H}^{G}([X]) = [G \times_H X]$. We will abuse notation and also write $\Tr_{H}^{G}(X) = G \times_H X$, for $X$ an $H$-set, so that $\Tr_{H}^{G}([X]) = [\Tr_{H}^{G}(X)]$. This map sends basis elements to basis elements since $[G \times_H (H/K)] = [G/K]$. When $G$ is a Dedekind group, $\Tr_{H}^{G}$ is injective. 

We describe $\Barks(G)$ in terms of the effective Burnside semiring $A^+(G)$. Let $I_H = \im \Tr_{H}^{G} \subseteq A^+(G)$, be the image of the transfer map. This is the summand of $A^+(G)$ on the basis elements of the form $[G/K]$ for $K \subseteq H$. It is closed under both addition and multiplication. Further, $I_H$ only depends on the conjugacy class of $H$ inside of $G$. We define
\[
\Barks(G) := \prod_{[H]} \Z[I_H],
\]
where $\Z[I_H]$ is the monoid ring of the (additive) commutative monoid $I_H$. Note that if $K$ is subconjugate to $H$ then $I_K \subseteq I_H$ and $\Z[I_K]$ is a subring of $\Z[I_H]$. 

Recall that $P(G)$ is a graded ring.  The ring $\Barks(G)$ admits a grading as well. The ring $\Z[I_H]$ is a polynomial ring with generators $x_{[K]} := [G/K] \in I_H$ for $K \subseteq H$ and where $[K] \in \Sub(G)/G$. We declare that the degree of $x_{[K]} \in \Z[I_H]$ is $|H/K|$. We write $\Barks(G)_m$ for the abelian group of homogeneous degree $m$ elements of $\Barks(G)$. When $G$ is Dedekind, we will write $x_K = x_{[K]}$ as $\Sub(G)/G = \Sub(G)$ in this case.

\begin{definition}
\label{character_formula}
    The marks homomorphism (or character map), $\chi: P(G) \rightarrow \Barks(G)$, for the Burnside--Tambara ring is given by 
    \[ \chi([X \xrightarrow{\varphi} Y])_{[H]} = \sum_{y \in Y^H} [G \times_H \varphi^{-1}(y)] = \sum_{y \in Y^H} [\Tr_{H}^{G} \varphi^{-1}(y)], \]
    regarding $\varphi^{-1}(y)$ as an $H$-set and where the sum takes place in $\Z[I_H]$. 
\end{definition}

To see that $\chi$ is a well-defined map, consider the situation where $[H] = [H']$ so $H' = gHg^{-1}$ for some $g \in G$. Multiplication by $g$ gives a bijection $Y^H \xrightarrow{\cong} Y^{gHg^{-1}}$. For $y \in Y^H$, $\varphi^{-1}(y)$ is an $H$-set and $\varphi^{-1}(gy)$ is an $H'$-set and the $G$-sets $G \times_H \varphi^{-1}(y)$ and $G \times_{H'} \varphi^{-1}(gy)$ are isomorphic via $[\ell,x] \mapsto [\ell g^{-1},gx]$. This shows that $\chi([X\xrightarrow{\varphi} Y])_{[H]} = \chi([X\xrightarrow{\varphi} Y])_{[H']}$. It then suffices to consider what happens when we take different representatives of the same bispan
$[X_1\xrightarrow{\varphi_1} Y_1] = [X_2 \xrightarrow{\varphi_2} Y_2]$. Let $\psi_Y \colon Y_1\to Y_2$ and $\psi_X \colon X_1\to X_2$ be $G$-equivariant isomorphisms that witness the equality of bispans. Because $\psi_Y$ is a $G$-equivariant isomorphism, it restricts to an isomorphism between $Y_1^H$ and $Y_2^H$. Similarly, $\psi_X$ restricts to an $H$-set isomorphism $\phi_{1}^{-1}(y) \xrightarrow{\cong} \phi_{2}^{-1}(\psi_X(y))$ for $y \in Y_{1}^{H}$. This, in addition to the relation $\varphi_2 \psi_X = \psi_Y\varphi_1$, allows us to conclude that 
\begin{align*}
\sum_{y \in Y_1^H}[\Tr_{H}^G \varphi_1^{-1}(y)] &= \sum_{y \in Y_1^H}[\Tr_{H}^G (\psi_X^{-1}\varphi_2^{-1}\psi_Y)(y)] \\ 
&= \sum_{\psi_Y(y) \in Y_2^H}[\Tr_{H}^G (\psi_X^{-1}\varphi_2^{-1})(\psi_Y(y))]\\
&=\sum_{\psi_Y(y)\in Y_2^H}[\Tr_{H}^G \varphi_2^{-1}(\psi_Y(y))]\\
&= \sum_{y \in Y_2^H}[\Tr_{H}^G \varphi_2^{-1}(y)].
\end{align*}
We conclude that $\chi$ is well-defined. 

Note that, without the transfer in the definition of $\chi$, the map would not be well-defined as distinct ways of conjugating $H$ to $H'$ in $G$ may result in different isomorphisms between $A^+(H)$ and $A^+(H')$. Also note that $\chi$, when restricted to $A(G) = P(G)_0$, recovers the marks homomorphism of \cref{classicalmarks}.

\begin{proposition} \label{prop:charmap}
    The character map $\chi$ is a graded ring homomorphism.
\end{proposition}

\begin{proof}
Let $[H] \in \Sub(G)/G$. Note that $\chi([\emptyset \rightarrow \emptyset])_{[H]} = 0$ and $\chi([\emptyset \rightarrow *])_{[H]} = \Tr^G_H([\emptyset])  = [\emptyset]$, which is $1$ in $\Z[I_H]$, so that $\chi$ respects the additive and multiplicative identities.

    To see that $\chi$ is additive, let $[X \xrightarrow[]{\varphi} Y]$ and $[X' \xrightarrow[]{\varphi'} Y']$
    be two elements of the Burnside--Tambara ring. We compute 
    \begin{align*} 
    \chi([X \amalg X' \xrightarrow{\varphi \amalg \varphi'} Y \amalg Y' ])_{[H]}
    &= \sum_{z \in (Y \amalg Y')^H} [G \times_H \left( \varphi \amalg \varphi'\right)^{-1}(z)]\\
    &= \sum_{y \in Y^H} [G \times_H \varphi^{-1}(y)] + \sum_{y' \in Y'^H} [G \times_H \varphi'^{-1}(y')]\\
    &= \chi ([X \xrightarrow{\varphi} Y ])_{[H]} + \chi ([X' \xrightarrow{\varphi'} Y' ])_{[H]}.
    \end{align*}
    
    For multiplicativity, we have 
    \begin{align*}
        &\chi ([(X \times Y') \amalg (X' \times Y) \xrightarrow{(\varphi \times \id_{B'}) + (\id_B \times \varphi')} Y \times Y'])_{[H]}\\ 
        =& \sum_{(y,y') \in (Y \times Y')^H} [ G \times_H ((\varphi \times \id_{B'}) + (\id_B \times \varphi'))^{-1}(y,y')] \\
        =& \sum_{(y,y') \in (Y \times Y')^H} [G \times_H \varphi^{-1}(y)] \times [G \times_H \varphi'^{-1}(y')] \\
        =& \left(\sum_{y \in Y^H} [G \times_H \varphi^{-1}(y)] \right) \cdot \left(\sum_{y' \in Y'^H} [G \times_H \varphi'^{-1}(y')] \right) \\
        =& \chi([X \xrightarrow{\varphi} Y])_{[H]} \cdot \chi([X' \xrightarrow{\varphi'} Y'])_{[H]}.
    \end{align*}
    
    Now assume that $[X \xrightarrow{\phi} Y] \in P(G)_m$ is homogeneous of degree $m$ so that for each $y \in Y$, we have $|\phi^{-1}(y)| = m$. Thus for $H \subseteq G$ and $y \in Y^H$, the $H$-set $\phi^{-1}(y)$ has cardinality $m$. The grading on $\Barks(G)$ is such that the degree of $[\Tr_{H}^{G}\phi^{-1}(y)] \in \Z[I_H]$ is $m$. 
\end{proof}

\begin{example}
\label{character_of_a_basis_element}
    We compute the characters of the additive basis elements of $P(G)$ described in \cref{basis_elements_are_canonical_quotients}. Recall that these are the bispans of the form
    $$\left[\coprod_{i=1}^n G/K_i \xrightarrow{\psi} G/L\right],$$
    where each $K_i\subseteq L$ and each summand $\psi_i$ of $\psi$ is the canonical quotient map $G/K_i\to G/L$ given by $xK_i\mapsto xL$. To apply the character formula of \cref{character_formula}, we need to know the $M$-fixed points of the codomain of $\psi$ and the $M$-set $\psi^{-1}(yL)$ for $yL\in (G/L)^M$ for every $[M]\in \Sub(G)/G$. The $M$-fixed point set $(G/L)^M$ consists of precisely those $yL \in G/L$ for which $y^{-1}My\subseteq L$. The fiber of an $M$-fixed point $yL$ under the summand $\psi_i$ is those $xK_i$ for which $xL = yL$, i.e. $x = yl$ for some $l\in L$, or equivalently $xK_i \in yL/K_i$. As an $M$-set, this admits an orbit decomposition given by the double coset formula
    $$yL/K_i \cong\coprod_{MylK_i \in M \backslash yL / K_i}{M/(({}^{yl} \! K_i) \cap M)}.$$
    Applying the character formula, we get
    \begin{align*}
        \chi[\psi]_{[M]}
        &= \chi\left( \left[ \coprod_{i=1}^{n} G/K_i \xrightarrow[]{\psi} G/L \right] \right)_{[M]}\\
        &= \sum_{\substack{yL \in G/L \\ y^{-1}My \subseteq L}}\left[\Tr^G_M(\psi^{-1}(y))\right]\\
        &= \sum_{\substack{yL \in G/L \\ y^{-1}My \subseteq L}}\left[\Tr^G_M\left(\coprod_{i=1}^{n} yL/K_i \right)\right]\\
        &= \sum_{\substack{yL \in G/L \\ y^{-1}My \subseteq L}}\left[\Tr^G_M\left(\coprod_{i=1}^{n} \coprod_{MylK_i \in M \backslash yL / K_i}{M/(({}^{yl} \! K_i) \cap M)} \right)\right]\\
        &= \sum_{\substack{yL \in G/L \\ y^{-1}My \subseteq L}}\prod_{i=1}^{n} \prod_{MylK_i \in M \backslash yL / K_i}\left[\Tr^G_M\left({M/(({}^{yl} \! K_i) \cap M)} \right)\right]\\
        &= \sum_{\substack{yL \in G/L \\ y^{-1}My \subseteq L}}\prod_{i=1}^{n} \prod_{MylK_i \in M \backslash yL / K_i}x_{[({}^{yl} \! K_i) \cap M]},
    \end{align*}
    where the fifth line followed from the fact that transfers preserve coproducts, the coproduct is the addition operation in the monoid $I_H$, and that turns into the product operation in the monoid ring $\Z[I_H]$.
    
    In particular, if $M$ is not subconjugate to $L$, this is the empty sum and hence $\chi[\psi]_{[M]} = 0$. If $[M] = [L]$, then we end up summing over $yL \in N_G(L)/L$ and there is only one double coset as $MylK_i = L$. Further, the transfer $\Tr_L^G(L/{}^{yl}\!K_i)$ is isomorphic to the transfer $\Tr_L^G(L/K_i)$. Therefore, when $[M] = [L]$ we have
    $$\chi[\psi]_{[L]} = |W_G(L)| \prod_{i=1}^n x_{[K_i]},$$
    where $W_G(L) = N_G(L)/L$ is the Weyl group of $L\subseteq G$.
\end{example}

\begin{proposition}
\label{basis_for_barks}
There is a canonical bijection between the degree $m$ monomial basis of the factor of $\Barks(G)_m$ corresponding to $[H] \in \Sub(G)/G$ and the collection of (multi)sets
\[
\{[K_1], \ldots, [K_n] \mid [K_i] \in \Sub(G)/G, [K_i] \leq [H], \text{ and } \sum_{i=1}^{n} |H|/|K_i| = m\}. 
\]
\end{proposition}
\begin{proof}
The monomial associated to the set
\[
\left\{[K_1], \ldots, [K_n] \;\middle |\; [K_i] \in \Sub(G)/G, [K_i] \leq [H], \text{ and } \sum_{i=1}^{n} |H|/|K_i| = m\right\}
\]
is $\prod_{i=1}^nx_{[K_i]}$ in the factor of $\Barks(G)$ corresponding to $[H]$, which has degree $m$.
\end{proof}

\begin{proposition}\label{character_map_rational_iso}
    If $G$ is a Dedekind group, then the map $\chi \colon P(G) \rightarrow \Barks(G)$ is injective and a rational isomorphism. 
\end{proposition}
\begin{proof}
    We check that this is true for the degree $m$ component of $\chi$ for each $m \in \N$. By \cref{dedekind_grps_character_iso_lemma,basis_for_barks}, $P(G)_m$ and $\Barks(G)_m$ are free abelian groups of the same rank. Explicitly, for $G$ Dedekind, there is a bijection $h$ from the basis of $P(G)_m$ to the basis of $\Barks(G)_m$ given by
    $$\left[\coprod_{i=1}^n G/K_i \xrightarrow{\psi} G/L\right] \mapsto \left(\prod_{i=1}^n x_{K_i}\text{ concentrated in the $L$-factor}\right).$$
    where $\sum_{i=1}^n [L:K_i] = m$. We choose an ordering on these bases and show that the associated matrix of $\chi_m$ is upper triangular with nonzero determinant. For the basis of $P(G)_m$, we choose an ordering $\prec$ in which, if $L$ is a proper subgroup of $L'$ and $\varphi$ and $\varphi'$ are bispans with codomains $G/L$ and $G/L'$ respectively, then $[\varphi] \prec [\varphi']$. We then use the bijection $h$ to transfer $\prec$ to an ordering on the basis of $\Barks(G)_m$.

    We claim that the matrix of $\chi_m$ with respect to these ordered bases is upper triangular, i.e. if $[\psi]$ is a basis element of $P(G)_m$ as above and $b$ is a basis element of $\Barks(G)_m$ with $[\psi] \prec h^{-1}(b)$, then the $([\psi], b)$-entry of the matrix is zero. If $[\psi] \prec h^{-1}b$, then either $b$ is a monomial concentrated in the $M$-factor with $M$ not contained in $L$, or $M = L$ with $h([\psi]) \neq b$. We now use the special cases noted at the end of \cref{character_of_a_basis_element}: if $M$ is not contained in $L$, then $\chi[\psi]_{M} = 0$, and if $M = L$, then $\chi[\psi]_{M}$ is a non-zero multiple of $h([\psi])$ and hence doesn't have $b$ as a summand. This also tells us that the $([\psi], h([\psi]))$-entry of the $\chi_m$ matrix is non-zero for all basis elements $[\psi]$ and hence all diagonal entries of this matrix are non-zero.
\end{proof}

Similarly to the table of marks, we can also describe the image of the marks homomorphism for $P(G)$ in a table of marks, where our rows are labeled by the generators of $P(G)$ as an $A(G)$-algebra and our columns are labeled by the generators of $\Barks(G)$.

\begin{example} \label{ex:cpcharacter}
    Using the generators described in \cref{big_burnside_ring_Cp_example_computation}, we get the following table of characters for $C_p$:
\begin{table}[H] 
\begin{center}
{\renewcommand{\arraystretch}{1.2}
\begin{tabular}{lcc}
         \toprule
         $ $
         &  $\Z[x_e]$ 
         & $\Z[x_e, x_{C_p}]$
         \\ \midrule
         $x =[C_p/C_p \to C_p/C_p]$ & $x_e$ & $x_{C_p}$  \\
			$n = [C_p/e\to C_p/C_p]$ & $x_e^p$ & $x_{e}$ 
        \\\bottomrule
\end{tabular} }
\end{center}
\end{table}
\end{example}

\begin{example} \label{ex:c4character}
    Using the generators described in \cref{big_burnside_ring_C4_example_computation}, we get the following table of characters for $C_4$:

    \begin{center}
		\begin{tabular}[H]{lccc}
            \toprule
			& $\Z[x_e]$ & $\Z[x_e,x_{C_2}]$ & $\Z[x_e,x_{C_2},x_{C_4}]$\\
			\midrule
            $s = [C_4/e \to C_4/C_2]$ & $2x_{e}^2$ & $2x_{e}$ & 0	\\
			$x = [C_4/C_4 \to C_4/C_4]$ & $x_e$ & $x_{C_2}$ & $x_{C_4}$ \\
			$m = [C_4/C_2\to C_4/C_4]$ & $x_e^2$ & $x_{C_2}^2$ & $x_{C_2}$ \\
			$n = [C_4/e \to C_4/C_4]$ & $x_e^4$ & $x_{e}^2$ &	$x_{e}$
            \\\bottomrule
		\end{tabular}
	\end{center}
One can check that this induces an isomorphism of rings between
\[
\Q \otimes A(C_4)[x,m,n,s]/(x^4v - nv, x^2u - mu,x^2v- mv,sv-2x^2v, su-2s, sm - sx^2, s^2-2un) 
\]
and $\Q \otimes \Barks(C_4)$, thereby verifying the presentation of \cref{big_burnside_ring_C4_example_computation}.
\end{example}

\subsection{Variants}

In this section we will discuss versions of $P(G)$ and $\Barks(G)$ depending on a choice of $G$-transfer system. We will also discuss an alternative construction of the character theory of the previous section making use of work of Thevenaz in \cite{Thevenaz1988}. The reader is free to skip this section as it is not critical to the rest of the paper. 

We recall from \cite[Def 3.4]{Rubin} the notion of a $G$-transfer system. 

\begin{definition}
Let $G$ be a finite group. A $G$-transfer system is a partial order $T$ on $\Sub(G)$ which refines the subset inclusion relation, is closed under conjugation, and is closed under intersections. The last condition means that if $K \leq H$ in $T$ and $J \subseteq H$ is a subgroup, then $K \cap J \leq J$ in $T$.
\end{definition}

For a $G$-transfer system $T$, let $P_T(G)$ be the Grothendieck ring of bispans of finite $G$-sets
\[
[X \xrightarrow{\varphi} Y]
\]
in which for all $x \in X$, the inclusion
\[
\Stab_G(x) \subseteq \Stab_G(\varphi(x)) 
\]
is in $T$. The properties of a $G$-transfer system ensure that products of maps of $G$-sets satisfying this condition still satisfy this condition. Note that $P_T(G) \subseteq P(G)$, with equality holding when $T$ is the complete transfer system.

Let $\Barks_T(G) = \prod_{[H]} \Z[I_{T,H}]$, where $I_{T,H} \subseteq A^+(G)$ is the subgroup generated by isomorphism classes of $G$-sets of the form $[G/K]$ for $K \subseteq H$ an inclusion in $T$. Since $I_{T,H} \subseteq I_H$, we may view $\Barks_T(G)$ as a subring of $\Barks(G)$. Following the proof of \cref{prop:charmap}, we have:

\begin{proposition}
The marks homomorphism for the Burnside--Tambara ring restricts to a map
\[
\chi \colon P_T(G) \to \Barks_T(G).
\]
\end{proposition}

We say that a subgroup $H \subseteq G$ is $G$-Dedekind if every subgroup of $H$ is normal in $G$. Since a $G$-Dedekind subgroup is normal and a subgroup of a $G$-Dedekind subgroup is also $G$-Dedekind, this notion gives rise to a transfer system. Let $T_{\lhd}$ be the transfer system containing $K \subseteq H$ if $H$ is $G$-Dedekind. 

\begin{example}
The $G$-Dedekind transfer system can be interesting. For instance, the metacyclic group $M_4(2)$ of order 16, which is a semidirect product of $C_2$ with $C_8$, has two distinct maximal proper $M_4(2)$-Dedekind subgroups.
\end{example}

Following the proof of \cref{character_map_rational_iso}, we have:
\begin{proposition}
Let $G$ be a finite group and let $T_{\lhd}$ be the $G$-Dedekind transfer system. The character map
\[
\chi \colon P_{T_{\lhd}}(G) \to \Barks_{T_{\lhd}}(G)
\]
is injective and rationally an isomorphism.
\end{proposition}

We now turn our attention to an alternative to the character map of \cref{prop:charmap}. Work of Thevenaz \cite{Thevenaz1988} studies the ring 
\[
\Twin_{P}(G) = \prod_{[H]} (P(H)/J_H)^{W_G(H)},
\]
where 
\[
J_H = \im \left(\bigoplus_{K \subset H} P(K) \xrightarrow{\sum_{K \subset H} \Tr_{K}^{H}} P(H) \right)
\]
and the Weyl group acts via conjugation. In \cite[Theorem 4.1 and Corollary 4.4]{Thevenaz1988}, Thevenaz proves a general result that implies that the map
\[
P(G) \to \Twin_P(G)
\]
given by restriction and taking the quotient induces a rational isomorphism
\[
\Q \otimes P(G) \to \Q \otimes \Twin_P(G).
\]
In this sense, Thevenaz's map is superior to the character map of \cref{character_formula}. The advantage of $\Barks(G)$ lies in the stellar algebraic properties it possesses. In fact, there is a close connection between Thevenaz's map and the character map of \cref{character_formula}.

\begin{proposition} \label{prop:thev}
There is an isomorphism of commutative rings
\[
P(H)/J_H \cong \Z[A^+(H)].
\]
\end{proposition}
\begin{proof}
Note that the map
\[
\chi_H \colon P(H) \to \Z[A^+(H)],
\]
induced by 
\[
\chi_H([X \xrightarrow{\phi} Y]) = \sum_{y \in Y^H} [\phi^{-1}(y)]
\]
is surjective, since every generator (elements $[X]$ of $A^+(H)$) is hit by a bispan of the form $[X \to *]$.

We will show that the kernel is $J_H$. Since all of the maps involved are additive, it suffices to prove this for $P^+(H)$. In \cref{ex:transfers}, we show that the transfer map $P^+(K) \to P^+(H)$ sends $[X \xrightarrow{\phi} Y]$ to $[\Tr_{K}^{H} X \xrightarrow{\Tr_{K}^{H} \phi} \Tr_{K}^{H} Y]$. We have that $J_H$ is contained in the kernel because the transfer of a bispan from a proper subgroup has no $H$-fixed points in its codomain. On the other hand, assume that $[X \to H/K]$ is in the kernel of $\chi_H$. Since $(H/K)^H = \emptyset$, we must have that $K$ is a proper subgroup of $H$. Now, by \cref{basis_elements_are_canonical_quotients}, $[X \to H/K] = \left[\coprod_{i=1}^n H/K_i \xrightarrow{\psi} H/K\right]$, where $\psi_i \colon H/K_i \to H/K$ is the canonical quotient. Since $K_i$ is a subgroup of $K$, which is a proper subgroup of $H$, this element is in the image of the transfer from $P^+(K)$.
\end{proof}

The ring of invariants $\Z[A^+(H)]^{W_G(H)}$ is rarely a polynomial ring. Let $S_H$ be the set of isomorphism classes of transitive $H$-sets and view this set as a $W_G(H)$-set. Decompose $S_H$ into transitive $W_G(H)$-sets $\coprod_i S_{H,i}$. The ring of invariants is a polynomial ring if the Cayley group homomorphism 
\[
W_G(H) \to \prod_i \Sigma_{|S_{H,i}|}
\]
is surjective. This and related properties could be interesting to pursue, but we will not do so here. 

Note that there is a canonical ring map
\[
\Twin_P(G) \to \Barks(G)
\]
given on each factor by the composite 
\[
(P(H)/J_H)^{W_G(H)} \cong \Z[A^+(H)]^{W_G(H)} \to \Z[A^+(H)] \to \Z[I_H], 
\]
making use of the isomorphism of \cref{prop:thev}. This ring map fits into a commutative triangle
\[
\xymatrix{P(G) \ar[r] \ar[rd]_-{\chi} &  \Twin_P(G) \ar[d] \\ & \Barks(G),}
\]
in which the horizontal map is Thevenaz's map. When $G$ is Dedekind, the $W_G(H)$-action on $\Z[A^+(H)]$ is trivial and Thevenaz's map agrees with \cref{character_formula}. This provides an alternative description of \cref{character_formula} and alternative proof, making use of the work of Thevenaz, of \cref{character_map_rational_iso}.

\section{Composition of Bispans}
\label{sec:Composition of Bispans}

In this section we focus on composition of bispans of $G$-sets. Since the effective Burnside--Tambara semiring $P^+(G)$ is the semiring of endomorphisms of $G/G$ in the category of bispans, composition gives a third operation that is distinct from multiplication and addition. To describe how this third operation interacts with the ring structure, we review composition in $\P^G$.

\subsection{The dependent product}

Composition of bispans is really a multistep process. It
starts with forming an exponential diagram. A key ingredient in these diagrams is the dependent product. 

\begin{definition}
    Let $f \colon X \to Y$ be a morphism of finite $G$-sets. 
    The dependent product functor along $f$
    \[
        \prod_f \colon \Fin^G_{/X} \to \Fin^G_{/Y}
    \]
    is the right adjoint to the pullback functor 
    \(
        f^* \colon \Fin^G_{/Y} \to \Fin^G_{/X}.
    \)
\end{definition}

More explicitly, given $\ell \colon A \to X$ in $\Fin^G_{/X}$, the dependent product of $\ell$ along $f \colon X \to Y$ can be modeled as $\prod_f A \xrightarrow{p} Y$, where 
\[
    \prod_f A \coloneqq 
        \big\{ 
            {(y, \sigma)} 
            \ \big\vert\  
            y \in Y,\ 
            \sigma \colon f^{-1}(y) \to A,\ 
            \ell \circ \sigma = \id_{f^{-1}(y)} 
        \big\} 
\]
and $p$ is the projection $(y, \sigma) \mapsto y$. 

The set $\prod_f A$ becomes a $G$-set via 
\[
    g \cdot (y, \sigma) = (gy, {}^g\!\sigma),
\]
where ${}^g\!\sigma$ is the conjugate of $\sigma$ by $g$: 
\[
    {}^g\!\sigma(x) = g\,\sigma(g^{-1}x).
\]
We think of the $G$-set $\prod_f A$ as a set of partial sections $\sigma$ of $\ell$, one over each fiber of $f$. 

\begin{example}
	If $Y = \ast$ is a singleton, then $\prod_f A$ is the set of sections of $\ell \colon A \to X$. 
\end{example}

\begin{example}
\label{dep_prod_along_id}
    If $f\colon X \to X$ is the identity map, then the dependent product $\prod_f A$ of $\ell: A\to X$ along $f$ is $\ell: A\to X$ itself. 
\end{example}

\begin{example}[{\cite[Proposition 2.3]{HM2019}}]
	Let $f \colon G/H \to \ast$ and $\ell \colon A \to G/H$. Then there is a canonical isomorphism of $G$-sets:
	\[
		\prod_f A \cong \Fin^H(G, \ell^{-1}(eH)). 
	\]	 
\end{example}

As stated previously, composing bispans also requires forming exponential diagrams. 

\begin{definition}
	Given morphisms $f \colon X \to Y$ and $\ell \colon A \to X$ in $\Fin^G$, the canonical exponential diagram generated by $f$ and $\ell$ is the commutative diagram 
	\[
		\begin{tikzcd}
			X 
				\ar[d, "f"]
				&
			A 
				\ar[l, "\ell"'] 
				& 
			X \times_Y \prod_f A 
				\ar[l, "e"']
				\ar[d, "\pi"]
				\ar[dll, phantom, "\llcorner" description, very near start]
				\\
			Y
				& 
				&
			\prod_f A,
				\ar[ll, "p"]
		\end{tikzcd}
	\]
	where $p$ is the projection $(y, \sigma) \mapsto y$, $e$ is the evaluation map $e(x,(y,\sigma)) = \sigma(x)$, and $\pi$ is projection onto the dependent product $\prod_f A$. Note that $\ell \circ e$ is projection onto $X$ and the outer rectangle of the diagram is a pullback. An exponential diagram is any diagram isomorphic to a canonical exponential diagram as above. 
\end{definition}

\subsection{Composition in the Burnside--Tambara Ring}
Here we define the composition product, and show that in the first coordinate the composition product has algebraic properties that we will make use of in our later calculation. 

Given bispans $[X \xleftarrow{} A \xrightarrow{} B \xrightarrow{} Y]$ and $[Y \xleftarrow{} C \xrightarrow{} D \xrightarrow{} Z]$, we may compose them by forming pullbacks and exponential diagrams as in \cref{bispan_composition}. In \cref{CompDiagram}, we specialize this to bispans with source and target a point.
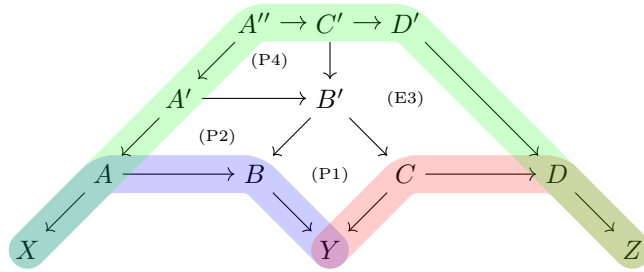
\begin{figure}[H]
\begin{tikzpicture}
	\node (X) at (1,0) {$X$};
	\node (A) at (2,1) {$A$};
	\node (B) at (4,1) {$B$};
	\node (Y) at (5,0) {$Y$};
	\node (C) at (6,1) {$C$};
	\node (D) at (8,1) {$D$};
	\node (Z) at (9,0) {$Z$};
	\node (B') at (5,2) {$B'$}; 
	\node (A') at (3,2) {$A'$}; 
	\node (D') at (6,3) {$D'$}; 
	\node (C') at (5,3) {$C'$}; 
	\node (A'') at (4,3) {$A''$}; 
	\draw[->] (A) -- (X);
	\draw[->] (A) -- (B);
	\draw[->] (B) -- (Y);
	\draw[->] (C) -- (Y);
	\draw[->] (C) -- (D);
	\draw[->] (D) -- (Z);
	\draw[->] (B') -- (B);
	\draw[->] (B') -- (C);
	\draw[->] (A') -- (A);
	\draw[->] (A') -- (B');
	\draw[->] (A'') -- (A');
	\draw[->] (A'') -- (C');
	\draw[->] (C') -- (B');
	\draw[->] (C') -- (D');
	\draw[->] (D') -- (D);
	\draw[white] (B') --node[black]{\tiny(P1)} (Y);
	\draw[white] (A') --node[black]{\tiny(P2)} (B);
	\node at (4.2,2.5) {\tiny(P4)};
	\draw[white] (D') --node[black]{\tiny(E3)} (C);
	\draw[blue,line width=5mm, rounded corners, line cap=round, opacity=0.2] 
		(1,0) -- (2,1) -- (4,1) -- (5,0);
	\draw[red,line width=5mm, rounded corners, line cap=round, opacity=0.2] 
		(5,0) -- (6,1) -- (8,1) -- (9,0);
	\draw[green,line width=5mm, rounded corners, line cap=round, opacity=0.2] 
		(1,0) -- (4,3) -- (6,3) -- (9,0);
\end{tikzpicture}
\caption{Composition of bispans. Given bispans $[X \xleftarrow{} A \xrightarrow{} B \xrightarrow{} Y]$ (highlighted in blue) and $[Y \xleftarrow{} C \xrightarrow{} D \xrightarrow{} Z]$ (highlighted in red), we form the composite bispan 
\(
	[Y \xleftarrow{} C \xrightarrow{} D \xrightarrow{} Z] 
		\circ 
	[X \xleftarrow{} A \xrightarrow{} B \xrightarrow{} Y]
\)
by first taking pullbacks (P1) and then (P2), then forming an exponential diagram (E3), then finally forming the pullback (P4). The resulting bispan $[X \leftarrow A'' \rightarrow D' \rightarrow Z]$  (highlighted in green) is the composite.}
\label{bispan_composition}
\end{figure}

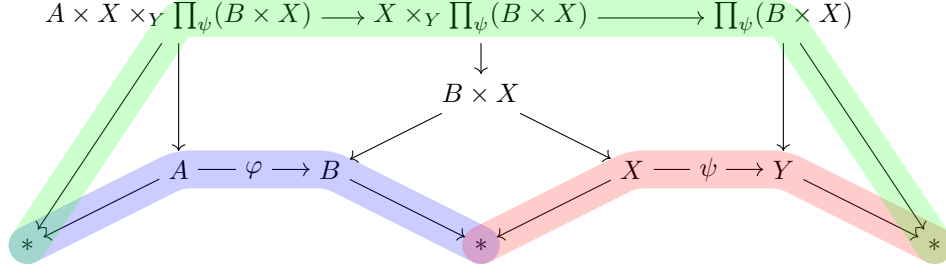
\begin{figure}[H]
\label{composition_product}
\begin{tikzpicture}
	\node (*1) at (0,0) {$\ast$};
	\node (A) at (2,1) {$A$};
	\node (B) at (4,1) {$B$};
	\node (*2) at (6,0) {$\ast$};
	\node (X) at (8,1) {$X$};
	\node (Y) at (10,1) {$Y$};
	\node (*3) at (12,0) {$\ast$};
	\node (BxX) at (6,2) {$B \times X$};
	\node (dp) at (10,3) {$\prod_\psi (B \times X)$};
	\node (pb1) at (6,3) {$X \times_Y \prod_\psi(B \times X)$};
	\node (pb2) at (2,3) {$A \times X \times_Y \prod_\psi(B \times X)$};
	\draw[->] (A) -- (*1);
	\draw[->] (A) --node[fill=white]{$\phi$} (B);
	\draw[->] (X) --node[fill=white]{$\psi$} (Y);
	\draw[->] (B) -- (*2);
	\draw[->] (X) -- (*2);
	\draw[->] (Y) -- (*3);
	\draw[->] (BxX) -- (B);
	\draw[->] (BxX) -- (X);
	\draw[->] (pb1) -- (BxX);
	\draw[->] (pb1) -- (dp);
	\draw[->] (dp) -- (Y);
	\draw[->] (pb2) -- (pb1);
	\draw[->] (pb2) -- (A);
	\draw[->] (pb2) -- (*1);
	\draw[->] (dp) -- (*3);	
	\draw[blue,line width=5mm, rounded corners, line cap=round, opacity=0.2] 
		(0,0) -- (2,1) -- (4,1) -- (6,0);
	\draw[red, line width=5mm, rounded corners, line cap=round, opacity=0.2]
		(6,0) -- (8,1) -- (10,1) -- (12,0);
	\draw[green,line width=5mm, rounded corners, line cap=round, opacity=0.2]
		(0,0) -- (2,3) -- (10,3) -- (12,0);
\end{tikzpicture}
\caption{Composition of bispans from the point to the point. The bispan $[A \xrightarrow{\varphi} B]$ is in blue, the bispan $[X \xrightarrow{\psi} Y]$ is in red, and their composite $[X \xrightarrow{\psi} Y] \circ [A \xrightarrow{\varphi} B]$ is in green.}
\label{CompDiagram}
\end{figure}

\begin{example}
    Consider the subgroup $P(G)_1 \subseteq P(G)$. There is an isomorphism of abelian groups $A(G) \cong P(G)_1$ given by sending $[X]$ to $[X \xrightarrow{\id_X} X]$. Since pullbacks and dependent products along identity maps are identity functors  (\cref{dep_prod_along_id}), we see that the composition of two such bispans reduces to the composition of spans and hence corresponds to the ordinary product on $A(G)$.
\end{example}

\begin{example} \label{ex:transfers}
Let $H \subseteq G$ and let $T_{H} = [G/H \leftarrow G/H \rightarrow G/H \rightarrow G/G]$. The transfer map $P(H) \to P(G)$ is the additive map induced by the composite
\[
P^+(H) \cong \P^G(\ast, G/H) \xrightarrow{T_{H} \circ (-)} \P^G(\ast, \ast) \cong P^+(G).
\]
The first isomorphism takes $[X \xrightarrow{\phi} Y] \in P^+(H)$ to the bispan
\[
[\ast \leftarrow G \times_H X \xrightarrow{G \times_H \phi} G \times_H Y \to G/H].
\]
Following \cref{bispan_composition} and making use of \cref{dep_prod_along_id}, composing this bispan with $T_H$ yields
\[
[\ast \leftarrow G \times_H X \xrightarrow{G \times_H \phi} G \times_H Y \to \ast] \in P^+(G). 
\]
Thus the transfer map $P(H) \to P(G)$ sends $[X \xrightarrow{\phi} Y]$ to $[\Tr_{H}^{G} X \xrightarrow{\Tr_{H}^{G}(\phi)} \Tr_{H}^{G}Y]$.
\end{example}

The category $\P^G$ has all finite products. 
\begin{proposition}[cf.~{\cite[proof of Proposition 2.12]{BH2018}}]
\label{products_in_P^G}
    Let $X_1,\dots, X_n \in \P^G$. Then their product is $\coprod_{i=1}^n X_i$, the coproduct of $G$-sets, with projection maps given by
    $$\coprod_{i=1}^n X_i \xleftarrow{} X_i \xrightarrow{\id} X_i \xrightarrow{\id} X_i,$$
    where the backwards arrow is the inclusion into the coproduct.
\end{proposition}
\begin{proof}
    We check the universal property of the product by verifying that postcomposing with the claimed projections induces an isomorphism $\P^G\left(Y, \coprod_{i=1}^n X_i\right) \cong \prod_{i=1}^n\P^G(Y, X_i)$. Given a bispan
    $$\left[Y \xleftarrow{\alpha} A \xrightarrow{\beta} B \xrightarrow{\gamma} \coprod_{i=1}^n X_i\right],$$
    its composite with the $i$th projection is
    $$[Y \xleftarrow{\alpha} (\gamma\beta)^{-1}(X_i) \xrightarrow{\beta} \gamma^{-1}(X_i) \xrightarrow{\gamma} X_i].$$
    This has an inverse which takes a collection of bispans
    $$\{[Y \xleftarrow{\alpha_i} A_i \xrightarrow{\beta_i} B_i \xrightarrow{\gamma_i} X_i]\}$$
    to the bispan
    $$\left[Y \xleftarrow{\sum_i \alpha_i} \coprod_{i=1}^n A_i \xrightarrow{\coprod_{i=1}^n \beta_i} \coprod_{i=1}^n B_i \xrightarrow{\coprod_{i=1}^n \gamma_i} \coprod_{i=1}^n X_i\right].$$
\end{proof}

We can use composition of bispans to understand the semiring structure on $\P^G(X,Y)$ described in \cref{proposition:semiring structure on bispans}. This comes from the semiring structure on the object $Y\in \P^G$ in which the addition and multiplication structure maps are given by
$$[Y\amalg Y\xleftarrow{\id} Y\amalg Y \xrightarrow{\id} Y\amalg Y \xrightarrow{\nabla} Y]$$
and
$$[Y\amalg Y\xleftarrow{\id} Y\amalg Y \xrightarrow{\nabla} Y \xrightarrow{\id} Y]$$
respectively. Specializing to the case $X=Y=*$, we get a semiring structure on $P^+(G)$, which agrees with the one introduced in \cref{proposition:semiring structure on bispans}.

For any two bispans $[X_1 \xrightarrow[]{\psi_1} Y_1], [X_2 \xrightarrow[]{\psi_2} Y_2] \in P(G)$, denote by $([\psi_1], [\psi_2])$ the bispan
    $$\left[*\xleftarrow{!} X_1 \amalg X_2 \xrightarrow{\psi_1\amalg \psi_2} Y_1 \amalg Y_2 \xrightarrow{!\amalg !} *\amalg *\right].$$
This notation is justified by \cref{products_in_P^G}, which says that $\ast\amalg\ast$ is the categorical product in $\P^G$, and the fact that postcomposing this bispan with the two projections of this product yields $[\psi_1]$ and $[\psi_2]$ respectively, so those are the ``coordinate functions'' of $([\psi_1],[\psi_2])$. In particular, for any $[\varphi]\in P(G)$, we have
$$([\psi_1],[\psi_2]) \circ [\varphi] = ([\psi_1] \circ [\varphi],[\psi_2]\circ [\varphi]),$$
as this is true of maps into any categorical product.

\begin{proposition}\label{ring_from_composition}
    Let $[X_1 \xrightarrow[]{\psi_1} Y_1], [X_2 \xrightarrow[]{\psi_2} Y_2] \in P(G)$ and let $([\psi_1], [\psi_2])$ be as above.
    \begin{enumerate}[(a)]
        \item\label{ring_from_composition a} Let $T_\nabla = [*\amalg * \xleftarrow{\id} *\amalg * \xrightarrow{\id} *\amalg * \xrightarrow{\nabla} *]$. Then
        $$[\psi_1]+[\psi_2] = T_\nabla\circ ([\psi_1], [\psi_2]).$$
        \item Let $N_\nabla = [*\amalg * \xleftarrow{\id} *\amalg * \xrightarrow{\nabla} * \xrightarrow{\id} *]$. Then
        $$[\psi_1][\psi_2] = N_\nabla \circ ([\psi_1], [\psi_2]).$$
        \item\label{ring_from_composition c} Let $[\emptyset \xrightarrow{\iota} Y] \in P(G)$ be an element of $A(G)$ viewed as a bispan. Then
        $$[\iota]\circ [\psi_1] = [\iota].$$
    \end{enumerate}
\end{proposition}
\begin{proof}
    \begin{enumerate}[(a)]
        \item The composite is computed via the following diagram.
            
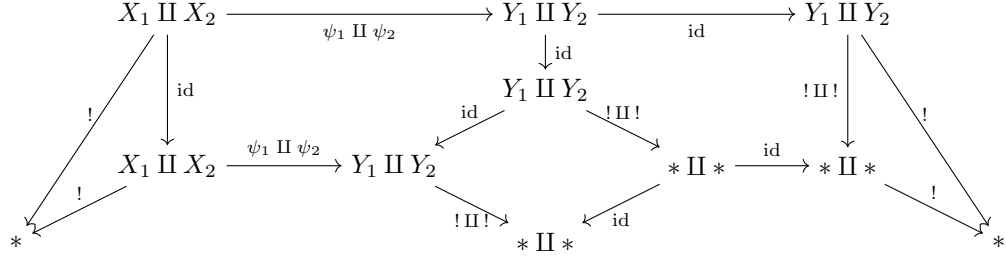
\begin{center}
            \begin{tikzpicture}
            	\node (*1) at (-1,0) {$\ast$};
            	\node (A) at (1,1) {$X_1\amalg X_2$};
            	\node (B) at (4,1) {$Y_1\amalg Y_2$};
            	\node (*2) at (6,0) {$*\amalg *$};
            	\node (X) at (8,1) {$*\amalg *$};
            	\node (Y) at (10,1) {$*\amalg *$};
            	\node (*3) at (12,0) {$\ast$};
            	\node (BxX) at (6,2) {$Y_1 \amalg Y_2$};
            	\node (dp) at (10,3) {$Y_1 \amalg Y_2$};
            	\node (pb1) at (6,3) {$Y_1 \amalg Y_2$};
            	\node (pb2) at (1,3) {$X_1 \amalg X_2$};
            	\draw[->] (A) --node[above, font=\scriptsize]{$!$} (*1);
            	\draw[->] (A) --node[above, font=\scriptsize]{$\psi_1\amalg\psi_2$} (B);
            	\draw[->] (X) --node[above, font=\scriptsize]{$\id$} (Y);
            	\draw[->] (B) --node[below, font=\scriptsize]{$!\,\amalg\,!$} (*2);
            	\draw[->] (X) --node[below, font=\scriptsize]{$\id$} (*2);
            	\draw[->] (Y) --node[above, font=\scriptsize]{$!$} (*3);
            	\draw[->] (BxX) --node[above, font=\scriptsize]{$\id$} (B);
            	\draw[->] (BxX) --node[above, font=\scriptsize]{$!\,\amalg\,!$} (X);
            	\draw[->] (pb1) --node[right, font=\scriptsize]{$\id$} (BxX);
            	\draw[->] (pb1) --node[below, font=\scriptsize]{$\id$} (dp);
            	\draw[->] (dp) --node[left, font=\scriptsize]{$!\,\amalg\,!$} (Y);
            	\draw[->] (pb2) --node[below, font=\scriptsize]{$\psi_1\amalg\psi_2$} (pb1);
            	\draw[->] (pb2) --node[right, font=\scriptsize]{$\id$} (A);
            	\draw[->] (pb2) --node[above, font=\scriptsize]{$!$} (*1);
            	\draw[->] (dp) --node[above, font=\scriptsize]{$!$} (*3);	
            \end{tikzpicture}
            \end{center}
            We see that it is precisely $[\psi_1]+[\psi_2]$.
        \item Similarly, $N_\nabla\circ ([\psi_1], [\psi_2])$ is computed by the following diagram.
            
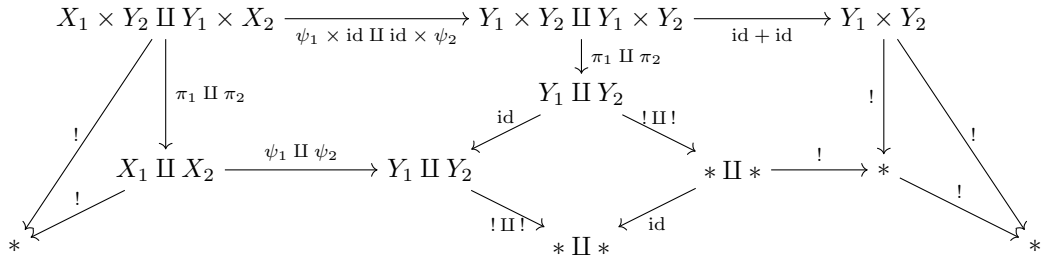
\begin{center}
            \begin{tikzpicture}
            	\node (*1) at (-1.5,0) {$\ast$};
            	\node (A) at (0.5,1) {$X_1\amalg X_2$};
            	\node (B) at (4,1) {$Y_1\amalg Y_2$};
            	\node (*2) at (6,0) {$\ast\amalg \ast$};
            	\node (X) at (8,1) {$\ast\amalg \ast$};
            	\node (Y) at (10,1) {$\ast$};
            	\node (*3) at (12,0) {$\ast$};
            	\node (BxX) at (6,2) {$Y_1 \amalg Y_2$};
            	\node (dp) at (10,3) {$Y_1 \times Y_2$};
            	\node (pb1) at (6,3) {$Y_1 \times Y_2 \amalg Y_1 \times Y_2$};
            	\node (pb2) at (0.5,3) {$X_1 \times Y_2 \amalg Y_1 \times X_2$};
            	\draw[->] (A) --node[above, font=\scriptsize]{$!$} (*1);
            	\draw[->] (A) --node[above, font=\scriptsize]{$\psi_1\amalg\psi_2$} (B);
            	\draw[->] (X) --node[above, font=\scriptsize]{$!$} (Y);
            	\draw[->] (B) --node[below, font=\scriptsize]{$!\,\amalg\, !$} (*2);
            	\draw[->] (X) --node[below, font=\scriptsize]{$\id$} (*2);
            	\draw[->] (Y) --node[above, font=\scriptsize]{$!$} (*3);
            	\draw[->] (BxX) --node[above, font=\scriptsize]{$\id$} (B);
            	\draw[->] (BxX) --node[above, font=\scriptsize]{$!\,\amalg\, !$} (X);
            	\draw[->] (pb1) --node[right, font=\scriptsize]{$\pi_1\amalg\pi_2$} (BxX);
            	\draw[->] (pb1) --node[below, font=\scriptsize]{$\id+\id$} (dp);
            	\draw[->] (dp) --node[left, font=\scriptsize]{$!$} (Y);
            	\draw[->] (pb2) --node[below, font=\scriptsize]{$\psi_1\times\id \amalg \id\times\psi_2$} (pb1);
            	\draw[->] (pb2) --node[right, font=\scriptsize]{$\pi_1\amalg\pi_2$} (A);
            	\draw[->] (pb2) --node[left, font=\scriptsize]{$!$} (*1);
            	\draw[->] (dp) --node[right, font=\scriptsize]{$!$} (*3);	
            \end{tikzpicture}
            \end{center}
        \item This follows from the fact that the dependent product of $\emptyset\to\emptyset$ along $\emptyset \to Y$ is the identity map $Y\to Y$.
    \end{enumerate}
\end{proof}
Specializing to bispans whose domain and codomain are both points, we get a composition 
\[
\circ \colon P^{+}(G) \times P^{+}(G) \to P^{+}(G).
\]
\begin{corollary}
\label{algebra_map_in_first_coord}
    Let $[\varphi] \in P^+(G)$. Then
    $$(-)\circ [\varphi]: P^+(G) \to P^+(G)$$
    is an $A(G)$-algebra map.
\end{corollary}
\begin{proof}
First, we check additivity. Let $[X_1 \xrightarrow{\psi_1} Y_1], [X_2 \xrightarrow{\psi_2} Y_2] \in P^+(G)$. By \cref{ring_from_composition}\eqref{ring_from_composition a}, we have
$$[\psi_1]+[\psi_2] = T_{\nabla} \circ ([\psi_1], [\psi_2]).$$
This lets us write
\begin{align*}
    ([\psi_1]+[\psi_2])\circ [\varphi] &= (T_\nabla\circ([\psi_1],[\psi_2])) \circ [\varphi]\\
    &= T_\nabla\circ(([\psi_1],[\psi_2]) \circ [\varphi])\\
    &= T_\nabla\circ([\psi_1]\circ[\varphi], [\psi_2]\circ [\varphi]) \\
    &= [\psi_1]\circ[\varphi] + [\psi_2]\circ[\varphi].
\end{align*}

The proof of multiplicativity can be obtained from the above by replacing all occurrences of addition with multiplication and $T_\nabla$ with $N_\nabla$. Lastly, we need to check that $(-) \circ [\varphi]$ fixes the scalars $A(G) \subseteq P(G)$. This is precisely \cref{ring_from_composition}\eqref{ring_from_composition c}.
\end{proof}

In particular, $(-) \circ [\varphi]$ extends to an $A(G)$-algebra map
\[
(-) \circ [\varphi] \colon P(G) \to P(G).
\]

\section{Plethories from Posets}

\subsection{Plethories}

The polynomial ring $\Z[x]$ admits much more structure than that of a commutative ring. In addition to the operations of addition and multiplication, there are the cooperations of coaddition $f \mapsto 1 \otimes f + f \otimes 1$ and comultiplication determined by $x \mapsto  x \otimes x$, as well as the composition of polynomials $(f,g) \mapsto f \circ g$. A plethory is an algebraic structure that codifies these operations and cooperations (see \cite{BorgerWieland2005}).

Let $R$ be a commutative ring and let $\Alg_R$ be the category of commutative $R$-algebras and $R$-algebra homomorphisms. An $R$-plethory is a comonad on $\Alg_R$ with the property that postcomposition with the forgetful functor to $\Set$ is corepresentable.

In more detail, assume $\Phi \colon \Alg_R \to \Alg_R$ is a lift of the functor $\Alg_R(B,-) \colon \Alg_R \to \Set$ to $\Alg_R$. This is the data of an $R$-biring structure on $B$, meaning that $B$ is equipped with the structure of an $R$-coalgebra in the category of commutative $R$-algebras (i.e. a coaddition $\Delta^+ \colon B \to B \otimes_R B$, comultiplication $\Delta^{\times} \colon B \to B \otimes_R B$, and co-$R$-linear structure $\beta \colon R \to \Alg_R(B,R)$, with an additive counit and multiplicative counit that satisfy the opposite of the axioms for a commutative $R$-algebra). Now assume that $\Phi_1$ and $\Phi_2$ are lifts of $\Alg_R(B_1,-)$ and $\Alg_R(B_2,-)$, respectively. It follows from \cite{TallWraith1970} that there is an $R$-biring $B_1 \odot_R B_2$ such that $\Phi_2 \circ \Phi_1$ is isomorphic to $\Alg_R(B_1 \odot_R B_2,-)$.

An $R$-plethory structure on $\Phi$ is the data of natural transformations $\Phi \to \Phi \circ \Phi$ and $\Phi \to 1$ making $\Phi$ into a comonad. This is equivalent to maps of $R$-birings $\circ \colon B \odot_R B \to B$ and $R[x] \to B$ making $B$ into a monoid for the $\odot_R$-product. A map of $R$-birings $\circ \colon B_1 \odot_R B_2 \to B_3$ is equivalent to a map of sets $\circ \colon B_1 \times B_2 \to B_3$ satisfying the relations in \cite[Section 1.3]{BorgerWieland2005}.

\subsection{Plethories from Posets}

\begin{definition}
\label{barks_definition_for_a_poset}
    Let $P$ be a finite poset. Define $\Marks(P)$ to be the ring $\prod_{p\in P} \Z$ and $\Barks(P)$ to be the product of polynomial rings $\prod_{p \in P} \Z[x_q \mid q \leq p]$.
\end{definition}
Note that $\Barks(P)$ is a $\Marks(P)$-algebra via the factor-wise inclusions $\Z \hookrightarrow \Z[x_q \mid q \leq p]$.

\begin{remark}
    We will care about the case $P = \Sub(G)/G$ for a finite group $G$, where the partial order is given by the subconjugacy relation. In the case where $G$ is the trivial group, $P$ is the singleton poset, $\Marks(P) = \Z$, and $\Barks(P) \cong \Z[x]$.
\end{remark}

Note that an algebra $R$ over $\Marks(P)$ canonically splits into a product $\prod_{p \in P} R_p$. A map of $\Marks(P)$-algebras is a product of maps of rings.

\begin{notation} \label{concentration notation}
    Let $R = \prod_{p \in P} R_p$ be a $\Marks(P)$-algebra. Given an element $f \in R$, we write $f_p$ for the $p^{th}$ coordinate of $f$. Given an element $a \in R_p$, we write $a^p$ for the element of $R$ whose $p$th factor is $a$ and all other factors are $0$.
\end{notation}

We will endow $\Barks(P)$ with the extra structure of a $\Marks(P)$-plethory. To do this, consider the functor $C \coloneqq \Alg_{\Marks(P)}(\Barks(P), -)$. Since each factor of $\Barks(P)$ is a polynomial ring, the set of $\Marks(P)$-algebra maps out of it admits a simple description:
    $$C(R) = C\left(\prod_{p\in P}R_p\right) \cong \prod_{p\in P}\prod_{q\leq p} R_p.$$
Explicitly, this isomorphism is given by
    \begin{equation}
    \begin{split}    
    \label{CR_bijection}
        f &\mapsto (f(x_q^p)_p)_{q\leq p} \\ 
        (x_q^p \mapsto a_{pq}^p) &\mapsfrom (a_{pq})_{q\leq p}.
    \end{split}
    \end{equation}
    
Here, and in the next two proofs, we write $x_p^q$ to mean $(x_p)^q \in \Barks(P)$ to ease the visual burden.
To make $\Barks(P)$ a $\Marks(P)$-biring, we need to give the set $C(R)$ the structure of a $\Marks(P)$-algebra functorially in $R$. The ring structure on $C(R)$ is the product ring structure. For the $\Marks(P)$-algebra structure, there are choices. We choose to rewrite $C(R)$ as
    $$\prod_{p\in P}\prod_{q \leq p} R_p \cong \prod_{q\in P}\prod_{q \leq p} R_p\cong \prod_{p\in P}\prod_{p \leq q} R_q,$$
where we have indexed over the other side of the inequality. Then the $\Marks(P)$-algebra structure map is the product of ring maps:
    $$\prod_{p\in P} \Z \xrightarrow{\prod_{p\in P} i} \prod_{p\in P} \prod_{p \leq q} R_q.$$
In the following results, we will be using this description of $C(R)$ instead, so we rewrite the isomorphism described in \cref{CR_bijection} as
\begin{equation}
    \begin{split}    
    \label{CR_bijection_rewrite}
        C(R) &\cong\prod_{p\in P}\prod_{p \leq q} R_q \\
        f &\mapsto (f(x_p^q)_q)_{p\leq q} \\
        (x_p^q \mapsto a_{pq}^q) &\mapsfrom (a_{pq})_{p\leq q}.
    \end{split}
    \end{equation}
\begin{lemma}
\label{cooperations in barks}
    The $\Marks(P)$-algebra structure on $\Alg_{\Marks(P)}(\Barks(P), R)$ described above is functorial in $R$ and hence makes $\Barks(P)$ into a $\Marks(P)$-biring with coaddition
    $$\Delta^+(x_q^p) = x_q^p\otimes 1 + 1 \otimes x_q^p,$$
    comultiplication
    $$\Delta^\times (x_q^p) = x_q^p \otimes x_q^p,$$
    and multiplicative counit
    $$\varepsilon^\times(x_q^p) = 1.$$
\end{lemma}
\begin{proof}

    The cooperations on $\Barks(P)$ correspond to the operations on $\Alg_{\Marks(P)}(\Barks(P),-)$ via the Yoneda correspondence. For example, given two maps $f,g \colon \Barks(P) \to R$, their sum is defined by first combining them into a single map
    $$\nabla(f\otimes g)\colon\Barks(P) \otimes_{\Marks(P)} \Barks(P) \to R$$
    from the coproduct and then precomposing with the coaddition map to get the sum
    $$\Barks(P) \xrightarrow{\Delta^+} \Barks(P) \otimes_{\Marks(P)} \Barks(P) \xrightarrow{\nabla(f\otimes g)} R.$$
    Therefore, it suffices to check that the claimed coaddition formula does induce the correct addition on the set of maps. Indeed, using the coaddition formula yields
    $$x_q^p \mapsto x_q^p\otimes 1 + 1 \otimes x_q^p \mapsto f(x_q^p)g(1) + f(1)g(x_q^p) = f(x_q^p)+g(x_q^p)$$
    which, under the isomorphism described in \cref{CR_bijection_rewrite}, corresponds to the coordinate-wise sum of the tuples corresponding to $f$ and $g$, which is how addition works in the product ring structure on $C(R)$. The proofs of the other two formulas are similar.
\end{proof}
Next we need to make $C$ into a comonad. First, we compute $C^2$:
\begin{align*}
    C^2(R) &\cong C\left(\prod_{p\in P}\prod_{p \leq q} R_q\right) \\
    &\cong \prod_{p\in P} \prod_{p \leq r} \prod_{r \leq q} R_q
\end{align*}
The comonadic counit $\epsilon_R \colon C(R) \to R$ is the factor-wise projection
    $$ \prod_{p\in P}\prod_{p \leq q} R_q \xrightarrow{\prod_p \pi_{p \leq p}} \prod_{p\in P} R_p, $$
and the comonadic comultiplication $\delta_R \colon C(R) \to C^2 (R)$ is the map
    $$ \prod_{p\in P}\prod_{p \leq q} R_q \rightarrow \prod_{p\in P} \prod_{p \leq r} \prod_{r \leq q} R_q $$
given by the product of diagonal maps
\[ (a_{pq})_{p \leq q} \mapsto (a_{pq})_{p \leq r \leq q}. \]

\begin{proposition}
\label{plethysm_in_barks}
    The structure maps described above endow $\Barks(P)$ with the structure of a $\Marks(P)$-plethory in which the plethysm $\circ \colon \Barks(P)\odot_{\Marks(P)}\Barks(P)\to \Barks(P)$ is given on the variables by
    $$x_u^q \circ x_p^r = \begin{cases}
        x_p^q, &\text{if $u=r$}\\
        0, &\text{otherwise.}
    \end{cases}$$
\end{proposition}
\begin{proof}
    The first part of the statement reduces to checking that $C$ is a comonad on $\Alg_{\Marks(P)}$. This is a straightforward check using the explicit descriptions of the structure maps, but we include the coassociativity check as an example:
    \[
    \begin{tikzpicture}
        \node (OuterUL) at (0,3) {$(a_{pq})_{p\leq q}$};
        \node (OuterDL) at (0,0)
        {$(a_{pq})_{p \leq r \leq q}$};
        \node (OuterDM) at (5,0)
        {$(a_{pq})_{p \leq r \leq s \leq q}$};
        \node (OuterDR) at (8,0)
        {$(a_{pq})_{p \leq s \leq r \leq q}$};
        \node (OuterUR) at (8,3)
        {$(a_{pq})_{p \leq r \leq q}$};
        \draw[|->] (OuterUL) -- (OuterDL);
        \draw[|->] (OuterUL) -- (OuterUR);
        \draw[|->] (OuterDL) -- (OuterDM);
        \draw[white] (OuterDM) -- node[color=black]{=} (OuterDR);
        \draw[|->] (OuterUR) -- (OuterDR);

        \node (InnerUL) at (1.5,2.25) {$C(R)$};
        \node (InnerUR) at (6,2.25) {$C^2(R)$};
        \node (InnerDL) at (1.5,1) {$C^2(R)$};
        \node (InnerDR) at (6,1) {$C^3(R)$};

        \draw[->] (InnerUL) -- node[above]{\small $\delta_R$} (InnerUR);
        \draw[->] (InnerUL) -- node[left]{\small $\delta_R$} (InnerDL);
        \draw[->] (InnerUR) -- node[right]{\small $\delta_{CR}$} (InnerDR);
        \draw[->] (InnerDL) -- node[below]{\small $C\delta_R$} (InnerDR);
    \end{tikzpicture}
    \]

    Next, we need to check that the formula for the plethysm in the statement corresponds to the comonadic comultiplication $\delta$. The plethysm $\Barks(P) \odot_{\Marks(P)} \Barks(P) \to \Barks(P)$ is obtained from $\delta$ by applying $\delta$ to the element $\id_{\Barks(P)} \in C(\Barks(P))$ to get an element of
    \begin{align*}
        C^2(\Barks(P)) &= \Alg_{\Marks(P)}\left(\Barks(P), \Alg_{\Marks(P)}(\Barks(P), \Barks(P))\right)\\
        &\cong \Alg_{\Marks(P)}(\Barks(P)\odot_{\Marks(P)} \Barks(P), \Barks(P)).
    \end{align*}
    The last step here looks like ``uncurrying'': it takes $(x \mapsto (y \mapsto z))$ to $(y\odot x \mapsto z)$ (note the change in the order). The comonadic comultiplication is defined by the diagram
    \[\begin{tikzcd}[column sep = small]
        C(\Barks(P))
        \ar{r}{\cong}
        \ar[dotted]{dd}
        &
        \prod_{p\in P} \prod_{p \leq q} \Z[x_s \mid s \leq q]
        \ar{d}{\delta_{\Barks(P)}}\\
        &
        \prod_{p\in P} \prod_{p \leq r} \prod_{r \leq q} \Z[x_s \mid s \leq q]\\
        C^2(\Barks(P))
        \ar{r}{\cong}
        &
        C\left(\prod_{p\in P}\prod_{p\leq q} \Z[x_s \mid s \leq q]\right). \ar{u}{\cong}
    \end{tikzcd}\]
    Chasing $\id_{\Barks(P)}$ around this diagram, we get
    \[\begin{tikzcd}[column sep = small, ampersand replacement = \&]
        \id_{\Barks(P)}
        \ar[mapsto]{r}
        \ar[mapsto]{dd}
        \&
        (x_p)_{p\leq q}
        \ar[mapsto]{d}\\
        \&(x_p)_{p\leq r\leq q}\ar[mapsto]{d}\\
        {\left(x_p^r \mapsto \left(x_u^q \mapsto \begin{cases}
            x_p^q, &\text{if $u = r$}\\
            0, &\text{if $u\neq r$} 
        \end{cases}\right)\right)\ar[mapsto]{r}}
        \&
        (x_p^r \mapsto ((x_p)_{r\leq q})^r).
    \end{tikzcd}\]
    Uncurrying this yields the map
    $$x_u^q \odot x_p^r \mapsto \begin{cases}
        x_p^q, &\text{if $u=r$}\\
        0, &\text{otherwise.}
    \end{cases} $$
\end{proof}
\begin{corollary}
\label{composing_with_arbitrary_elements_in_barks}
    For $p\leq q\in P$ and $f\in \Barks(P)$, we have
    $$x_p^q \circ f = (f_p)^q.$$
\end{corollary}
\begin{proof}
    First, note that the map $(x_p^q \circ -) \colon \Barks(P) \to \Barks(P)$ is a ring map for every $p \leq q \in P$. This follows from the coaddition and comultiplication formulas for $x_p^q$ in \cref{cooperations in barks}. For example, the additivity comes from
    \begin{align*}
        x_p^q \circ (a + b) &= (\Delta^+ x_p^q)(a,b) \\
        &= (x_p^q \otimes 1 + 1\otimes x_p^q)(a,b) \\
        &:= (x_p^q\circ a)(1\circ b) + (1\circ a)(x_p^q \circ b) \\
        &= x_p^q \circ a + x_p^q\circ b.
    \end{align*}
    The equality of the second and third lines is as in \cite[Equation 1.3.2]{BorgerWieland2005}. 
    
    Now, we write $f = \sum_{r\in P} (f_r)^r$, where $f_r$ is a polynomial in variables $x_s$ with $s\leq r$. Since $x_{p}^{q} \circ (-)$ is a ring map and 
       $$x_u^q \circ x_p^r \mapsto \begin{cases}
        x_p^q, &\text{if $u=r$}\\
        0, &\text{otherwise,}
       \end{cases}
       $$
       we have
    \begin{align*}
        x_p^q \circ f &= x_p^q \circ (\sum_{r\in P}(f_r)^r) \\
        &= x_p^q \circ (f_p)^p \\
        &= (f_p)^q.
    \end{align*}
\end{proof}
We note another property of this plethysm that we use later in the paper.
\begin{proposition}
\label{plethysm_depends_only_on_the_relevant_coordinate}
    Let $P$ be a finite poset and let $f, g\in \Barks(P)$. Then for any $p\in P$,
    $$(f\circ g)_p = ((1^pf) \circ g)_p = ((f_p)^p \circ g)_p,$$
    i.e., the function $(f\circ -)_p$ only depends on the $p^{th}$-coordinate of $f$.
\end{proposition}
\begin{proof}
    Since $\Barks(P)$ is a $\Marks(P)$-plethory, its plethysm $\circ$ is a $\Marks(P)$-algebra map in the first entry. We then have
    \begin{align*}
        (f\circ g)_p = (1^p(f\circ g))_p = ((1^pf) \circ g)_p
    \end{align*}
    and $1^pf = (f_p)^p$.
\end{proof}

\begin{remark}
\label{relative_barks}
    All of the above arguments go through relative to a fixed ring $R$, if we define
    $$\Marks(P,R) = \prod_{p\in P} R$$
    and
    $$\Barks(P,R) = \prod_{p\in P} R[x_q\mid q \leq p].$$
\end{remark}

\begin{remark}
The reader may notice that the results of this section generalize to a finite category $C$ in place of a finite poset. In this case,
\[
\Marks(C) = \prod_{c \in \mathrm{ob}C} \Z
\]
and
\[
\Barks(C) = \prod_{c \in \mathrm{ob}C} \Z[x_{f} \mid f \in \mathrm{ob} (C_{/c})]
\]
and the plethysm is determined by
\[
x_f \circ x_g = \begin{cases}
        x_{f \circ g}, &\text{if $f$ and $g$ are composable}\\
        0, &\text{otherwise.}
    \end{cases}
\]
\end{remark}

\begin{example}
    Let $p$ be a prime number. Consider the group $G = C_p$ and the finite poset $P = \Sub(C_p)/C_p = \Sub(C_p)$.
    By \cref{barks_definition_for_a_poset}, we have 
    \begin{equation*}
        \Barks(C_p) = \Z[x_e] \times \Z[x_e, x_{C_p}].
    \end{equation*} Now
    consider the two elements (see \cref{ex:cpcharacter})
    \begin{align*}
        \chi([C_p/e \to \ast]) &= (x_e^p, x_{C_p}) \\
        \chi([C_p/C_p \to \ast]) &= (x_e, x_e)
    \end{align*}
    of $\Barks(C_p)$. Here $x_e^p$ denotes the $p$th power of $x_e$. The composition product of these elements is
    \begin{align*}
        (x_e^p, x_{C_p}) \circ (x_e, x_e) &= ((x_e^p, 0) + (0, x_{C_p})) \circ (x_e, x_e)\\
        &= ((x_e^p, 0) \circ (x_e, x_e)) + ((0, x_{C_p}) \circ (x_e, x_e))\\
        &= ((x_e, 0) \circ (x_e, x_e))^p + ((0, x_{C_p}) \circ (x_e, x_e)),
    \end{align*}
    because plethysms are always ring maps in the first entry. Applying \cref{composing_with_arbitrary_elements_in_barks} to each summand, we obtain 
    \begin{equation*}
    (x_e^p, 0) + (0,x_e) = 
    (x_e^p, x_e). 
    \end{equation*}
\end{example}

\begin{example}
\label{composition_product_for_C4_example}
Let $G = C_4$ so that
\[
\Barks(C_4) = \Z[x_e] \times \Z[x_e,x_{C_2}] \times \Z[x_e,x_{C_2},x_{C_4}].
\]
Consider the element $s = [C_4/e \rightarrow C_4/C_2] \in P(C_4)$, which, by \cref{ex:c4character}, has character
\[ \chi(s) = (2x_e^2,2x_e,0) \in \Barks(C_4). \]   
Let $\bar{u} = (u_1, u_2, u_3)\in\Barks(C_4)$. The composite 
\begin{equation*}
    \chi(s) \circ \bar{u} = 
    (2x_e^2, 2x_e, 0) \circ (u_1, u_2, u_3). 
\end{equation*}
is given by 
\begin{align*}
    (2x_e^2, 2x_e, 0) \circ  (u_1, u_2, u_3)
    &= ((2x_e^2, 0, 0) + (0,2x_e,0)) \circ (u_1, u_2, u_3)\\
    &= ((2x_e^2, 0, 0) \circ (u_1, u_2, u_3)) + ( (0,2x_e,0) \circ (u_1, u_2, u_3))\\
    &= (2u_1^2, 0, 0) + (0, 2u_1, 0)\\
    &=(2u_1^2, 2u_1, 0).
\end{align*}
Note that we have made use of the fact that $u_1$ can be viewed as an element in the second factor of $\Barks(C_4)$.
\end{example}
The above examples demonstrate how straightforward it is to compute the composition product in the $\Marks(G)$-plethory $\Barks(G)$.

\section{Compatibility of Composition}

In this section, we show that the composition product on $P^+(G)$ described in \cref{CompDiagram} and the plethysm on $\Barks(G)$ defined in 
\cref{plethysm_in_barks} (for the poset $P = \Sub(G)/G$) are compatible through the character map defined in \cref{character_formula}. For all finite $G$, we show that the following equality holds in $\Barks(G)$:
\begin{equation}
\label{chi_respects_composition}
    \chi([X \xrightarrow{\psi} Y] \circ_P [A \xrightarrow{\varphi} B]) = \chi([X \xrightarrow{\psi} Y]) \circ_B \chi([A \xrightarrow{\varphi} B]),
\end{equation}
where $\circ_P$ is the composition product in $P^+(G)$ and $\circ_B$ is the composition product in $\Barks(G)$. The left-hand side here requires computing the character of a composite, for which we need an explicit description of the composite.

\subsection{Composites, explicitly} We compute composites in $P(G)$ and set up some notational conventions which are employed regularly throughout the rest of the section.

\begin{proposition}
	\label{alt_dep_prod}
	Let $\prod_f B \times X$ be the dependent product of the projection map $\ell \colon B \times X \to X$ along $f \colon X \to Y$. There is a bijection of sets
	\[
		\prod_f B \times X \cong \bigsqcup_{y \in Y} \Fin(f^{-1}(y),B).
	\]
    The $G$-action on the left determines the following $G$-action on the right:
    \[
    g(y, h) = (gy, gh(g^{-1}-)).
    \]   
\end{proposition}
\begin{proof}
    Recall that the dependent product on the left-hand side of the above is given by 
    \begin{equation*}
        \prod_f B \times X = \{ (y, \sigma) \mid y \in Y, \sigma \colon f^{-1}(y) \to B \times X, \ell \circ \sigma = \id_{f^{-1}(y)} \}.
    \end{equation*}
    Moreover, note that $\Fin(f^{-1}(y), B)$ is isomorphic to $\Fin(f^{-1}(y), l^{-1}(y))$. 
    Then, the desired isomorphism is obtained
    from the simple observation that for all $x \in f^{-1}(y)$, the equality $\ell(\sigma(x)) = x$ implies that $\sigma(x) \in \ell^{-1}(x)$ is isomorphic to $B$. In other words, each such section really has codomain $B$.
\end{proof}
\begin{notation}
\label{disjoint_union_vs_coprod}
    In \cref{alt_dep_prod}, we use the symbol $\bigsqcup$ instead of $\coprod$. Throughout this section, we maintain a careful distinction between the two. The coproduct symbol $\coprod$ denotes the usual coproduct of $G$-sets in which $G$ acts by acting on each summand separately. On the other hand, a disjoint union such as $\bigsqcup_{i\in I} X_i$ will mean that the sets $X_i$ are not necessarily closed under the action of $G$, but the entire disjoint union is. All disjoint unions that we will encounter in this section come from the description of the dependent product in \cref{alt_dep_prod}. Consequently, these disjoint unions will have the feature that $G$ acts on the indexing set $I$ to permute the summands, while also transforming the entries in each summand.
\end{notation}

\begin{lemma}
\label{formula_for_composition_product}
    Let $[X \xrightarrow[]{\psi} Y]$ and $[A \xrightarrow[]{\phi} B]$ be two elements of $P(G)$. Then their composition
    product is given by 
    \[
		[X \xrightarrow{\psi} Y] \circ [A \xrightarrow{\phi} B] = \left[ \bigsqcup_{x \in X} A \times \Fin(\psi^{-1}(\psi(x)) \setminus \{x\},B) \xrightarrow[]{\beta} \bigsqcup_{y \in Y} \Fin(\psi^{-1}(y),B)\right],
	\]
    where the $G$-action on the source is given by
    \[
g(x,a,f) = (gx,ga,{}^g\!f),
    \]
    the $G$-action on the target is given by
    \[
    g(y,f) = (gy, {}^g\!f),
    \]
    and
    \begin{equation*}
        \beta(x,(a,f)) = (\psi(x), \tilde{f} \colon \psi^{-1}(\psi(x)) \to B)
    \end{equation*}
    satisfying 
    \begin{equation*}
        \tilde{f}(t) = 
        \begin{cases}
        \varphi(a) &\text{ if } t = x\\
        f(t) &\text{ otherwise}.
        \end{cases}
    \end{equation*}
\end{lemma}
\begin{proof}
By \cref{alt_dep_prod}, the dependent product of the projection map $\pi_2 \colon B \times X \to X$ along $\psi \colon X \to Y$ is given by $\bigsqcup_{y \in Y}\Fin(\psi^{-1}(y),B)$. The remainder of the top row in \cref{CompDiagram} is obtained by taking successive pullbacks. The objects in the top row are easily seen to be pullbacks of the diagrams highlighted in blue and red.   
\end{proof}

To prove \cref{chi_respects_composition}, it suffices to check the case in which $[X\to Y]$ is a basis element of $P^+(G)$ since both composition products are additive in the first entry. By \cref{basis_elements_are_canonical_quotients}, we can assume that the basis elements are of the form 

\begin{equation*}
    \left[ \coprod_{i=1}^{n} G/K_i \xrightarrow[]{\sum\limits_{i=1}^{n} \psi_i} G/L \right],
\end{equation*}
with each $K_i\subseteq L$ and where each $\psi_i$ is the canonical quotient maps $G/K_i \to G/L$. We will make this assumption throughout this section.

\begin{notation}
    To make it easier to type check the expressions in this section, we use square brackets exclusively to mean ``equivalence class'' of subgroups (under conjugacy), $G$-sets (under isomorphism), or bispans (under isomorphism). Also, since the monoid ring construction on $A^+(G)$ turns the addition of $A^+(G)$ into multiplication, we have
    \begin{equation}
        \label{coproducts_to_products}
        \left[\coprod_{i\in I} X_i\right] = \prod_{i\in I}[X_i]
    \end{equation}
    for any finite collection of $G$-sets $\{X_i\}_{i\in I}$. Here, we follow the convention that putting square brackets around a $G$-set makes it into an element of the monoid ring $\mathbb Z[A^+(G)]$. In particular, the right side in the above equation is a product in the monoid ring and should not be interpreted as the isomorphism class of the product $G$-set. 
    
    Similarly, putting a subgroup $M$ of $G$ in the superscript of an expression will always mean taking $M$-fixed points and putting a conjugacy class $[M]$ in the superscript means a tuple concentrated in the $[M]$-coordinate (\cref{concentration notation}). 
\end{notation}
\begin{lemma}
\label{composition_product_of_basis_elt_with_arb_elt}
    Let $\{K_i\}_{1 \leq i \leq n}$ be a set of subgroups of $G$ contained in $L \subseteq G$ and let $$\left[\coprod_{i=1}^{n} G/K_i \xrightarrow[]{\psi} G/L\right]$$ be an element of $P^+(G)$ such that each summand $\psi_i \colon G/K_i \to G/L$ is the canonical quotient map. Additionally, let 
    $[A \xrightarrow[]{\varphi} B]$ be an arbitrary element of $P^+(G)$. Then
    the composition product $[\psi] \circ_P [\varphi]$
    is given by
    \begin{equation*}
        \left[ \coprod_{j=1}^{n} \bigsqcup_{xK_j \in G/K_j} 
        A \times \Fin\left(\left(\coprod_{i=1}^{n} xL/K_i\right) \setminus \{ (j, xK_j) \}, B\right) \xrightarrow[]{\beta} 
        \bigsqcup_{yL \in G/L} \Fin\left(\coprod_{i=1}^{n} yL/K_i, B\right) \right].
    \end{equation*}
    The $G$-action on the domain is given by 
    \begin{equation*}
        g(j,xK_j, (a, h)) = 
        (j, gxK_j, (ga, gh(g^{-1}-))),
    \end{equation*}
    where $1 \leq j \leq n$, $xK_j \in G/K_j$, $a \in A$, and $h \colon (\coprod_{i=1}^{n} xL/K_i) \setminus \{ (j, xK_j) \} \to B$. 
    The $G$-action on the codomain is
    given by
    \begin{equation*}
        g (yL, f) = (gyL, gf(g^{-1} -)),
    \end{equation*}
    where $yL \in G/L$ and $f \colon \coprod_{i=1}^{n} yL/K_i \to B$. 
    The map $\beta$ is given by
    $$\beta(j, xK_j, (a, f)) = \left(xL, \tilde f\right),$$
    where $\tilde{f}$ is determined by extending $f$ to all of $\coprod_{i=1}^n xL/K_i$ by defining $f(j, xK_j) = \varphi(a)$.
\end{lemma}
\begin{proof}
    For a given $yL \in G/L$, the fiber
    $\psi^{-1}(yL)$ is given by $\coprod_{i=1}^{n} yL/K_i$
    since each $\psi_i$ is the canonical quotient map. 
    An application of \cref{formula_for_composition_product} 
    yields 
    \begin{equation*}
        \left[ \bigsqcup_{(j, xK_j) \in \coprod_{j=1}^{n} G/K_j}
        A \times \Fin\left(\left(\coprod_{i=1}^{n} xL/K_i \right)\setminus \{ (j, xK_j) \}, B\right)
        \to
        \bigsqcup_{yL \in G/L}\Fin\left(\coprod_{i=1}^{n} yL/K_i, B\right) \right].
    \end{equation*}
    Rewriting the domain as follows
    \begin{equation*}
        \left[ \coprod_{j=1}^{n} \bigsqcup_{xK_j \in G/K_j} 
        A \times \Fin\left(\left(\coprod_{i=1}^{n} xL/K_i \right) \setminus \{ (j, xK_j) \}, B\right) \to 
        \bigsqcup_{yL \in G/L} \Fin\left(\coprod_{i=1}^{n} yL/K_i, B\right) \right]
    \end{equation*}
    yields the desired equality. The descriptions of the $G$-actions and $\beta$ also come from \cref{formula_for_composition_product}.
\end{proof}

\subsection{Proving compatibility} We now begin our proof of \cref{chi_respects_composition} in the case where $[X\xrightarrow{\psi} Y]$ is as in \cref{composition_product_of_basis_elt_with_arb_elt}. We do this by computing the $[M]$-coordinate of both sides of the equation for $[M] \in \Sub(G)/G$ and showing that those are equal. 

We first compute the left-hand side of \cref{chi_respects_composition}, which involves computing the character of the composition product recorded in \cref{composition_product_of_basis_elt_with_arb_elt}. Recall the definition of the character map given in \cref{character_formula}. When evaluating the character of a specific element $[X \xrightarrow[]{\varphi} Y]$ of the Burnside--Tambara ring at the $[M]$- coordinate, the ingredients include the $M$-fixed points of the codomain $Y$ and fibers of the morphism $\varphi$ at these $M$-fixed points. The following lemmas compute each of these ingredients in the setting of \cref{composition_product_of_basis_elt_with_arb_elt}. 

\begin{lemma}
\label{J_fixed_pts_description}
    Let $\left[\coprod_{i=1}^{n} G/K_i \xrightarrow[]{\psi} G/L \right]$ and $[A \xrightarrow{\varphi} B]$ be as in \cref{composition_product_of_basis_elt_with_arb_elt} and 
        \begin{equation*}
        \left[ \coprod_{j=1}^{n} \bigsqcup_{xK_j \in G/K_j} 
        A \times \Fin\left(\left(\coprod_{i=1}^n xL/K_i \right) \setminus \{ (j,xK_j) \}, B\right) \xrightarrow{\beta}
        \bigsqcup_{yL \in G/L} \Fin\left(\coprod_{i=1}^{n} yL/K_i, B\right) \right]
    \end{equation*}
    be the composition product computed in the previous lemma. Then the $M$-fixed point set of the codomain is isomorphic to
    \begin{equation*}
        \bigsqcup_{\substack{yL \in (G/L) \\ y^{-1}My \subseteq L}} \Fin^M \left(\coprod_{i=1}^n yL/K_i, B\right)
        \cong \bigsqcup_{\substack{yL \in (G/L) \\ y^{-1}My \subseteq L}} \prod_{i=1}^{n} \prod_{MylK_i \in M\backslash yL / K_i} B^{ ({}^{yl} \! K_i) \cap M},
    \end{equation*}
    where $B^{({}^{yl} \! K_i) \cap M}$ denotes the $(({}^{yl} \! K_i) \cap M)$-fixed point set of $B$. 
\end{lemma}
\begin{proof}
    Note that an element $yL \in G/L$ is fixed by left-multiplication by $M$ if and only if $myL = yL$ for all $m \in M$, or in other words, $y^{-1}My \subseteq L$. This shows that
    \begin{equation*}
        \left( \bigsqcup_{yL \in G/L} \Fin\left(\coprod_{i=1}^{n} yL/K_i, B\right) \right)^{M} = 
        \bigsqcup_{\substack{yL \in (G/L) \\ y^{-1}My \subseteq L}}
        \left(\Fin\left(\coprod_{i=1}^{n} yL/K_i, B\right)\right)^{M}.
    \end{equation*}
    Now we focus our attention on one summand, i.e. let $y\in G$ such that $y^{-1}My \subseteq L$. This makes each $yL/K_i$ an $M$-set.
    The action of an element $m \in M$ on a function $f \colon \coprod_{i=1}^{n} yL/K_i \to B$ yields the function $mf(m^{-1} -)$. It
    follows that the $M$-fixed points of $\Fin\left( \coprod_{i=1}^{n} yL/K_i, B\right)$ are precisely the 
    $M$-equivariant maps and this provides an identification 
    \begin{equation*}
        \bigsqcup_{\substack{yL \in (G/L) \\ y^{-1}My \subseteq L}}
        \left(\Fin\left(\coprod_{i=1}^{n} yL/K_i, B\right)\right)^{M}
        = 
        \bigsqcup_{\substack{yL \in (G/L) \\ y^{-1}My \subseteq L}}
        \Fin^M\left(\coprod_{i=1}^{n} yL/K_i, B\right).
    \end{equation*}
     When $y M y^{-1} \subseteq L$, $yL/K_i$ as an $M$-set has an orbit decomposition given by the double coset formula
     \begin{equation*}
         yL/K_i \cong \coprod_{MylK_i \in M \backslash yL / K_i}
         M/(({}^{yl} \! K_i) \cap M).
     \end{equation*}
     This allows us to further expand our expression, yielding isomorphisms
     \begin{align*}
         \bigsqcup_{\substack{yL \in (G/L) \\ y^{-1}My \subseteq L}}
        \Fin^M\left(\coprod_{i=1}^{n} yL/K_i, B\right)
        &\cong \bigsqcup_{\substack{yL \in (G/L) \\ y^{-1}My \subseteq L}}
        \Fin^M\left(\coprod_{i=1}^{n} \coprod_{MylK_i \in M \backslash yL / K_i}
         M/(({}^{yl} \! K_i) \cap M), B\right) \\
        &\cong \bigsqcup_{\substack{yL \in (G/L) \\ y^{-1}My \subseteq L}} \prod_{i=1}^{n} \prod_{MylK_i \in M \backslash yL / K_i}
        \Fin^M\left( M/(({}^{yl} \! K_i) \cap M), B\right) \\
        &\cong \bigsqcup_{\substack{yL \in (G/L) \\ y^{-1}My \subseteq L}} \prod_{i=1}^{n} \prod_{MylK_i \in M \backslash yL / K_i} B^{ ({}^{yl} \! K_i) \cap M}
     \end{align*}
     as we claimed.
\end{proof}

\begin{notation} The iterated product in the conclusion of \cref{J_fixed_pts_description} recurs throughout our computation so to ease the visual burden we set 
    \begin{equation}
    \label{what_is_BA}
      \BA_y := \prod_{i=1}^{n} \prod_{M ylK_i \in M \backslash yL / K_i} B^{({}^{yl} \! K_i) \cap M},
    \end{equation}
    for $y\in G$. We may then write the set of $M$-fixed points in \cref{J_fixed_pts_description} as
    $$\bigsqcup_{\substack{yL\in G/L \\ y^{-1}My \subseteq L}} \BA_y.$$
    For an element $\overline{b} \in \BA_y$, we will denote its $(i,MylK_i)$-coordinate by
        $$b_{(i,MylK_i)} \in B^{({}^{yl} \! K_i)\cap M} \cong \Fin^M(M/({}^{yl}K_i \cap M), B).$$
\end{notation}

\begin{notation}
\label{notate_gamma_i_s}
Let $yL \in (G/L)^M$ be an element of the $M$-fixed points of $G/L$. For $1 \leq i \leq n$, the fiber $\psi_i^{-1}(yL) =yL/K_i$ is an $M$-set. We choose a set of generators for this $M$-set, i.e. we choose one representative from each $M$-orbit in this $M$-set. From that choice of generators, we can obtain a set function
    \begin{equation*}
       \gamma_i(l) \colon L \to M
    \end{equation*}
    such that for any given $l \in L$, the left coset $\gamma_i(l)ylK_i$ is the chosen generator in the $M$-orbit $MylK_i$. Note that the functions $\gamma_i$ are not uniquely determined, but we make a choice and fix it for the rest of this section.
\end{notation}

The reason for making the choice in \cref{notate_gamma_i_s} is that the isomorphism of \cref{J_fixed_pts_description} is not canonical, but fixing a choice of $M$-set generators for each $yL/K_i$ fixes the isomorphism
$$\mymacro_y \colon \BA_y \cong \Fin^M \left( \coprod_{i=1}^n yL/K_i, B\right)$$
given by
\begin{align*}
   \overline{b} &\mapsto ((i, ylK_i) \mapsto \gamma_i(l)^{-1} b_{i, MylK_i})
\end{align*}
i.e. sending $\overline{b}$ to the $M$-equivariant map which sends each chosen $M$-set generator $\gamma_i(l)ylK_i$ to $b_{i, MylK_i}$. All ensuing references to the isomorphism in \cref{J_fixed_pts_description} will use this specific isomorphism.
  
With these remarks out of the way, we are now in place to compute the fiber of a given $M$-fixed point as in \cref{J_fixed_pts_description} under $\beta$ as in \cref{composition_product_of_basis_elt_with_arb_elt}.

\begin{lemma}
\label{fiber_of_a_fixed_pt}
    Let $\beta$ be the map as in $\cref{composition_product_of_basis_elt_with_arb_elt}$. Let $zL\in G/L$ with $z^{-1}Mz \subseteq L$ and $\overline{b}\in \mathbb{A}_z$, so that $(zL, \mymacro_y(\overline{b})) $ is an element of the $M$-fixed point set computed in \cref{J_fixed_pts_description}. Then $\beta^{-1}(zL, \mymacro_y(\overline{b}))$ is isomorphic, as an $M$-set, to 
    \begin{equation*}
        \coprod_{j=1}^{n} \{(zlK_j, a) \in zL/K_j \times  A : a \in \varphi^{-1}(\gamma_j(l)^{-1} b_{(j, MzlK_j)})\}.
    \end{equation*}
    The action of $M$ on each summand of this coproduct is the diagonal action $m(zlK_j, a) = (mzlK_j, ma)$.
\end{lemma}

\begin{proof}
    Recall that $\beta$ is the map
    \begin{align*}
        \coprod_{j=1}^n \bigsqcup_{xK_j\in G/K_j} A \times \Fin\left(\left(\coprod_{i=1}^n xL/K_i\right) \setminus \{ (j,xK_j)\}, B\right) &\xrightarrow{\beta}\bigsqcup_{yL\in G/L} \Fin\left(\coprod_{i=1}^n yL/K_i, B\right) \\
        (j, xK_j, (a, f)) &\mapsto (xL, \tilde{f}),
    \end{align*}
    where $\tilde{f}$ is the same as $f$ but extended by defining
        $$\tilde{f}(j, xK_j) = \varphi(a).$$
    Then, $\beta^{-1}(zL, \mymacro_y(\overline{b}))$ consists of those $(j,xK_j, (a,f))$ for which $(xL, \tilde f)=(zL, \mymacro_y(\overline{b}))$. It follows that $xL$ agrees with  $zL$ and $\tilde{f}$ agrees with $\mymacro_y(\overline{b})$. The former condition is equivalent to $xK_j = zlK_j$ for some $l \in L$. Using the explicit description of $\mymacro_y$, the latter condition is
        $$\tilde{f}(i, ylK_i) = \gamma_i(l)^{-1} b_{(i, MylK_i)}$$
    for all $(i, ylK_i) \in \coprod_{i=1}^{n} yL/K_i$.
    Since $\tilde{f}$ is an extension of $f$, this completely determines $f$, making it redundant since we have fixed a fixed point $(zL, \mymacro_y(\overline{b}))$. However, it gives us an extra condition by evaluating at the new point $(j, xK_j)$:
        $$\varphi(a) = \gamma_j(l)^{-1} b_{(j, MzlK_j)}.$$
    Therefore, elements of $\beta^{-1}(yL, \mymacro_y(\overline{b}))$ are in bijective correspondence with triples $(j, zlK_j, a)$ with $l\in L$ and $a \in \varphi^{-1}(\gamma_j(l)^{-1} b_{j, MzlK_j})$. The $M$-action on these triples is the one on the domain of $\beta$ but ignoring the redundant coordinate $f$, so $m(j, zlK_j, a) = (j, mzlK_j, ma)$. Since $m$ does not act on the index $j$, we get the coproduct decomposition as claimed.
\end{proof}
We now put the previous two lemmas together in the following: 
\begin{lemma}
\label{character_of_composite}
    Let $\left[\coprod_{i=1}^{n} G/K_i \xrightarrow[]{\psi} G/L \right]$ and $[A \xrightarrow{\varphi} B]$ be as in \cref{composition_product_of_basis_elt_with_arb_elt}. Then the $[M]$-coordinate for the character of the composite
\begin{equation*}
    \chi\left[ \psi \circ_P \varphi\right]_{[M]}
\end{equation*}
is
\begin{equation*}
    \sum_{\substack{yL \in G/L \\ y^{-1}My \subseteq L}}  \sum_{\overline{b} \in \BA_y}
    \prod_{j=1}^{n} \left[\Tr^G_M \left(  \{(ylK_j, a) \in yL/K_j \times  A : a \in \varphi^{-1}(\gamma_j(l)^{-1} b_{(j, MylK_j)})\} \right)\right],
\end{equation*}
where $\BA_y$ is as in \cref{what_is_BA}. 
\end{lemma}
\begin{proof}
    By \cref{composition_product_of_basis_elt_with_arb_elt}, we need to compute
    \begin{align*}
         \chi \left( \left[ \coprod_{i=1}^{n} \bigsqcup_{xK_i \in G/K_i} 
        A \times \Fin(xL/K_i \setminus \{ xK_i \}, B) \to 
        \bigsqcup_{yL \in G/L} \Fin\left(\coprod_{i=1}^{n} yL/K_i, B\right)\right] \right)_{[M]}.\\
    \end{align*}
    Applying the formula for the character as defined in \cref{character_formula} along with the description of the $M$-fixed points of the codomain as $\coprod_{\substack{yL \in G/L \\ y^{-1}My \subseteq L}}\BA_y$ turns this into
    \begin{equation*}
        \sum_{\substack{yL \in G/L \\ y^{-1}My \subseteq L}}  \sum_{\overline{b} \in \BA_y} [\Tr^G_M (\beta^{-1}(yL, \mymacro_y(\overline{b})))]. 
    \end{equation*}
    Using the description of $\beta^{-1}(yL, \mymacro_y(\overline{b}))$ in \cref{fiber_of_a_fixed_pt}, we get
    \begin{equation*}
    \sum_{\substack{yL \in G/L \\ y^{-1}My \subseteq L}}  \sum_{\overline{b} \in  \BA_y} \left[\Tr^G_M \left( \coprod_{j=1}^{n} \{(ylK_j, a) \in yL/K_j \times  A : a \in \varphi^{-1}(\gamma_j(l)^{-1} b_{(j, MylK_j)})\} \right)\right],
    \end{equation*}
    which yields the desired expression after we pull the coproduct out past the transfer and use \cref{coproducts_to_products}.
\end{proof}

We further simplify this expression by identifying the $M$-set that appears in it as the transfer of an $(({}^{yl}\!K_j) \cap M)$-set.
\begin{lemma}
\label{technical_lemma}
    Let $\left[\coprod_{i=1}^{n} G/K_i \xrightarrow[]{\psi} G/L \right]$ and $[A \xrightarrow{\varphi} B]$ be as in \cref{composition_product_of_basis_elt_with_arb_elt}. Let $1 \leq j \leq n$,
$yL \in (G/L)^M$, and $\overline{b} \in \BA_y$. Then the $M$-set 
 \begin{equation*} 
     \{(ylK_j, a) \in yL/K_j \times A: a \in 
     \varphi^{-1}(\gamma_j(l)^{-1} b_{(j, MylK_j)})\}     
     \end{equation*}
     is isomorphic to the $M$-set
     \begin{equation*}
     \coprod_{MylK_j \in 
     M \backslash yL / K_j} \Tr^M_{({}^{yl}\!K_j) \cap M} ( \varphi^{-1}(b_{(j, Myl K_j)})).
 \end{equation*}      
\end{lemma}
\begin{proof}
    Consider the $M$-set
    \begin{equation*}
        \{(ylK_j, a) \in yL/K_j \times A: a \in 
     \varphi^{-1}(\gamma_j(l)^{-1} b_{(j, MylK_j)})\}.
    \end{equation*}
    Using the orbit decomposition of $yL/K_j$ as an $M$-set, this decomposes into 
    \begin{equation*}
        \coprod_{MylK_j \in M \backslash yL / K_j} \{ (mylK_j, a) \in M yl/K_j \times A: a \in 
     \varphi^{-1}(\gamma_j(l)^{-1} b_{(j, MylK_j)})\}.
    \end{equation*}
    Therefore, it suffices to produce an isomorphism
    \begin{equation*}
        \{ (mylK_j, a) \in Myl/K_j \times A: a \in \varphi^{-1}(\gamma_j(l)^{-1} b_{(j, MylK_j)})\}
        \cong
        \Tr^M_{({}^{yl}\!K_j) \cap M} ( \varphi^{-1}(b_{(j, Myl K_j)}))
    \end{equation*}
    for a fixed double coset $MylK_j$. In other words, we are focusing our attention on one $M$-orbit of $yL/K_j$, namely $Myl/K_j$. Without loss of generality, we may assume that $l$ is such that $ylK_j$ is the chosen generator of the orbit $Myl/K_j$, so that $ylK_j = \gamma_j(l)ylK_j$. This implies that $\gamma_j(l)\in {}^{yl}\!K_j$, so $\gamma_j(l)^{-1}b_{(j, MylK_j)} = b_{(j, MylK_j)}$. Now, we construct an explicit $M$-equivariant isomorphism from the left-hand side to the right-hand side using the definition of the transfer,
    \begin{align*}
        \Phi \colon \{ (mylK_j, a) \in Myl/K_j \times A: a \in 
     \varphi^{-1}(b_{(j, MylK_j)})\}
     &\to 
      M {\displaystyle{\times_{({}^{yl}\!K_j) \cap M}}} ( \varphi^{-1}(b_{(j, Myl K_j)}) )\\
      (mylK_j, a) &\mapsto (m, m^{-1}a).
    \end{align*}
    To see that this map is well defined, let $m_1ylK_j = m_2ylK_j$ so that $m_1^{-1}m_2 \in ({}^{yl} \! K_j) \cap M$. It follows that 
    \begin{align*}
        \Phi(m_2ylK_j, a) &= (m_2, m_2^{-1}a) \\
        &= (m_1m_1^{-1}m_2, m_2^{-1}a) \\
        &= (m_1, m_1^{-1}m_2m_2^{-1}a) \\
        &= (m_1, m_1^{-1}a) \\
        &= \Phi(m_1ylK_j, a).
    \end{align*}
    It remains to show that the map is $M$-equivariant. To see that, let $\tilde{m} \in M$ and note that 
    \begin{equation*}
        \tilde{m} \Phi(mylK_j, a) 
        = \tilde{m} (m, m^{-1}a)
        = (\tilde{m}m, m^{-1}a),
    \end{equation*}
    whereas
    \begin{equation*}
        \Phi(\tilde{m}mylK_j, \tilde{m}a) = 
        (\tilde{m}mylK_j, (\tilde{m}m)^{-1}\tilde{m}a) = (\tilde{m}m, m^{-1}a).
    \end{equation*}
    This demonstrates the desired $M$-equivariance. We see that this assignment is an isomorphism because it admits an inverse: 
    \begin{align*}
        \zeta \colon M \times_{({}^{yl}\!K_j) \cap M} \left( \varphi^{-1}(b_{(j, Myl K_j)}) \right) &\to \{ (mylK_j, a) \in Myl/K_j \times A: a \in 
     \varphi^{-1}(b_{(j, MylK_j)})\}\\
      (m, a) &\mapsto (mylK_j, ma). 
    \end{align*}
    To see that this map is well-defined let 
    $c = ylkl^{-1}y^{-1} \in ({}^{yl} \! K_j) \cap M$. We know that 
    \begin{equation*}
      (mylkl^{-1}y^{-1}, a) \mapsto (mylK_j, mylkl^{-1}y^{-1}a),   
    \end{equation*}
    while on the other hand
    \begin{equation*}
        (m, ylkl^{-1}y^{-1}a) \mapsto (mylK_j, mylkl^{-1}y^{-1}a).
    \end{equation*}
\end{proof}

Putting everything together, we have computed every coordinate of the left-hand side of \cref{chi_respects_composition}:
\begin{equation}
\label{M_coord_of_char_of_composite}
\chi[\psi \circ_P \varphi]_{[M]} = \sum_{\substack{yL \in G/L \\ y^{-1}My \subseteq L}}  \sum_{\overline{b} \in \BA_y}
    \prod_{j=1}^{n}\prod_{MylK_j \in 
     M \backslash yL / K_j} \left[\Tr^G_{({}^{yl}\!K_j) \cap M}  ( \varphi^{-1}(b_{(j, Myl K_j)}))\right].
\end{equation}
We now compute the coordinates of the right-hand side of \cref{chi_respects_composition}.

\begin{lemma}
\label{RHS_of_compatibility_of_chi}
    Let $\left[\coprod_{i=1}^{n} G/K_i \xrightarrow[]{\psi} G/L \right]$ and $[A \xrightarrow{\varphi} B]$ be as in \cref{composition_product_of_basis_elt_with_arb_elt}. Then the $[M]$-coordinate for the composite of the characters is
    \begin{align*}
        (\chi[\psi]\circ_{B} \chi[\varphi])_{[M]} =
        \sum_{\substack{yL \in G/L \\ y^{-1}My \subseteq L}}  \sum_{\overline{b} \in \BA_y}
        \prod_{j=1}^{n}\prod_{MylK_j \in 
         M \backslash yL / K_j} \left[\Tr^G_{({}^{yl}\!K_j) \cap M}  ( \varphi^{-1}(b_{(j, Myl K_j)}))\right],
    \end{align*}
    where $\BA_y$ is as in \cref{what_is_BA}.
\end{lemma}
\begin{proof}
    By \cref{plethysm_depends_only_on_the_relevant_coordinate},
    \begin{equation}
    \label{coordinate_reduction}
    \left( \chi[\psi] \circ_{B} \chi[\varphi] \right)_{[M]} = \left( (\chi[\psi]_{[M]})^{[M]} \circ_{B} \chi[\varphi] \right)_{[M]}.
    \end{equation}
    In \cref{character_of_a_basis_element}, we computed
    \begin{align*}
        \chi[\psi]_{[M]} &= \sum_{\substack{yL \in G/L \\ y^{-1}My \subseteq L}}\prod_{i=1}^{n} \prod_{MylK_i \in M \backslash yL / K_i}x_{[({}^{yl} \! K_i) \cap M]}.
    \end{align*}
    Plugging this expression into the right-hand side of \cref{coordinate_reduction}, we get
    \begin{align*}
        \left( \chi[\psi] \circ_{B} \chi[\varphi] \right)_{[M]}
        &= \left(\left(\sum_{\substack{yL \in G/L \\ y^{-1}My \subseteq L}}\prod_{i=1}^{n} \prod_{MylK_i \in M \backslash yL / K_i}x_{[({}^{yl} \! K_i) \cap M]}\right)^{[M]} \circ \chi[\varphi]\right)_{[M]}\\
        &= \sum_{\substack{yL \in G/L \\ y^{-1}My \subseteq L}}\prod_{i=1}^{n} \prod_{MylK_i \in M \backslash yL / K_i}\left(x_{[({}^{yl} \! K_i) \cap M]}^{[M]} \circ \chi[\varphi]\right)_{[M]}\\
        &= \sum_{\substack{yL \in G/L \\ y^{-1}My \subseteq L}}\prod_{i=1}^{n} \prod_{MylK_i \in M \backslash yL / K_i}\chi[\varphi]_{[({}^{yl} \! K_i) \cap M]}\\
        &= \sum_{\substack{yL \in G/L \\ y^{-1}My \subseteq L}}\prod_{i=1}^{n} \prod_{MylK_i \in M \backslash yL / K_i}\sum_{b\in B^{({}^{yl}\!K_i) \cap M}} [\Tr_{({}^{yl}\!K_i) \cap M}^G (\varphi^{-1}(b))].\\
    \end{align*}
    Here, we first leveraged that $(-)^{[M]}$ and $(-\circ \chi[\varphi])$ are both rings maps, then the definition of $\circ_B$, and then the definition of $\chi$. Finally, distributing the two products over the innermost sum, we get
    $$\sum_{\substack{yL \in G/L \\ y^{-1}My \subseteq L}}  \sum_{\overline{b} \in \BA_y}
    \prod_{j=1}^{n}\prod_{MylK_j \in 
     M \backslash yL / K_j} \left[\Tr^G_{({}^{yl}\!K_j) \cap M}  ( \varphi^{-1}(b_{(j, Myl K_j)}))\right].$$
\end{proof}

\begin{theorem}
\label{compatibility_with_composition_in_barks}
    Let $[X \to Y]$ and $[A \to B]$ be two elements of $P^+(G)$.
    Then 
    \begin{equation*}
        \chi([X \to Y] \circ_P [A \to B]) = \chi([X \to Y]) \circ_B \chi([A \to B]).
    \end{equation*}
\end{theorem}
\begin{proof}
    If $[X \to Y]$ is of the form $[\coprod_i G/K_i \xrightarrow{\psi}G/L]$ where each $K_i \subseteq L$ and $\psi$ is the sum of the canonical quotient maps, we see that this equation is true coordinate-wise: the left-hand side is given in \cref{M_coord_of_char_of_composite} and the right-hand side is calculated in \cref{RHS_of_compatibility_of_chi}. These bispans form an additive basis for $P(G)$. Since the character map is additive and both composition products $\circ_P$ and $\circ_B$ are additive in the first coordinate (\cref{algebra_map_in_first_coord}), the equality holds in $P(G)$.
\end{proof}


\section{Cooperations on the Burnside--Tambara Ring} 
In this section, we show that, for a prime $p$, 
$P(C_p)$ is an $A(C_p)$-plethory. We also discuss some consequences of \cref{compatibility_with_composition_in_barks}: we show that $P(C_4)$ does not admit the structure of an
$A(C_4)$-plethory compatible with composition of bispans and that, for $G$ Dedekind, $\Q \otimes P(G)$ is a $\Q \otimes A(G)$-plethory.

\begin{proposition} \label{Cpplethory}
Let $p$ be a prime number. Then the Burnside--Tambara ring $P(C_p)$ admits the structure of an $A(C_p)$-plethory compatibly with the composition of bispans. 
\end{proposition}
\begin{proof}
    We use the explicit presentations of $A(C_p)$ and $P(C_p)$:
        $$A(C_p) \cong \frac{\mathbb Z[t]}{(t^2 - pt)}$$
        $$P(C_p) \cong \frac{A(C_p)[x,n]}{(tn - tx^p)},$$
    with notation as in \cref{big_burnside_ring_Cp_example_computation}.
    Let $C := \Alg_{A(C_p)}(P(C_p), -)$ be the functor corepresented by $P(C_p)$. The above presentation gives us a natural isomorphism
    \begin{equation*}
       CR \cong \{(a,b) \in R^2 : tb = ta^p\}, 
    \end{equation*}
    where $R$ is an arbitrary $A(C_p)$-algebra.
    To give an $A(C_p)$-biring structure on $P(C_p)$ is equivalent to giving a natural $A(C_p)$-algebra structure to the above set.
    
    Given pairs $(a,b), (c,d) \in CR$, their sum is given by 
    \begin{equation*}
        \left(a+c,b+d+\frac{t}{p}\sum_{i=1}^{p-1} \binom{p}{i}a^ic^{p-i} \right).
    \end{equation*}
    Note that the sum in the second coordinate is always divisible by $p$. The additive inverse of any pair $(a,b)$ is given by 
    \[ \left( -a, -b-\frac{t}{p} \sum_{i=1}^{p-1} \binom{p}{i} a^i(-a)^{p-i} \right). \]
    The multiplication on $CR$ is given by coordinate-wise multiplication.

    The $A(C_p)$-algebra structure on $CR$ is given by the ring map
    $A(C_p) \rightarrow CR$ determined by 
    \begin{equation*}
     0 \mapsto (0,0), \, 1 \mapsto (1,1), \, \text{ and } t \mapsto (t,p+(p^{p-1}-1)t).    
    \end{equation*}
    This shows that $P(C_p)$ is an $A(C_p)$-biring. In fact, this biring structure is the restriction of the $\Marks(C_p)$-biring structure on $\Barks(C_p)$; the cooperations on $\Barks(C_p)$, when restricted to the image of $\chi$, happen to land in the image of $\chi\otimes \chi: P(C_p)\otimes_{A(C_p)} P(C_p) \to \Barks(C_p) \otimes_{\Marks(C_p)} \Barks(C_p)$.

    To see that $P(C_p)$ is an $A(C_p)$-plethory, we verify that $C$ carries the structure of a comonad. The counit is given by the projection map
    \[ CR \twoheadrightarrow R \]
    onto the first factor, and the comultiplication map $CR \rightarrow C^2R$ is given by 
    \[ (a,b) \mapsto \left( (a,b),(b,b^p) \right). \]
    It can then be checked either by direct computation or using the fact that the biring structure is the one restricted from $\Barks(C_p)$ that these are both $A(C_p)$-algebra maps.
\end{proof}

\begin{proposition} \label{C4notplethory}
The Burnside--Tambara ring $P(C_4)$ does not admit the structure of an $A(C_4)$-plethory compatibly with the composition of bispans. 
\end{proposition}

\begin{proof}
Let $s = [C_4/e \rightarrow C_4/C_2] \in P(C_4)$ whose character is given by 
\[ \chi(s) = (2x_e^2,2x_e,0) \in \Barks(C_4). \]
We show that there is no element $\Delta^\times s \in P(C_4) \otimes_{A(C_4)} P(C_4)$ satisfying the equation 
\begin{equation}
    \label{comult_identity} (\Delta^\times s)(a,b) = s \circ (a \cdot b)
\end{equation}
for all $a, b \in P^+(C_4)$, and hence there is no comultiplication on $P(C_4)$ that is compatible with the composition product.

Assume there is such an element $\Delta^\times s$. Since the character map $\chi$ respects the ring operations as well as the composition product by \cref{compatibility_with_composition_in_barks}, we can apply it to both sides of \cref{comult_identity} to get the equation
\begin{equation}
    \label{chi_comult_identity} ((\chi\otimes \chi)(\Delta^\times s))(\bar{u},\bar{v}) = \chi(s) \circ (\bar{u} \cdot \bar{v})
\end{equation}
for all $\bar{u}, \bar{v} \in\chi(P^+(C_4))$. We exploit this to compute the $C_2$-coordinate $((\chi \otimes \chi)(\Delta^\times s))_{C_2}$ and arrive at a contradiction by showing that no such element can be in the image of $(\chi\otimes\chi)$.
Computing the right-hand side of \cref{chi_comult_identity}, using \cref{composition_product_for_C4_example}, we get
\begin{align*}
    \chi(s) \circ (\bar{u} \cdot \bar{v})
    &= (2x_e^2, 2x_e, 0) \circ (u_1v_1,u_2v_2,u_3v_3) \\
    &= (2u_1^2v_1^2, 2u_1v_1, 0).
\end{align*}
Now for the left-hand side, note that since $\Barks(C_4) = \Z[x_e] \times \Z[x_e, x_{C_2}] \times \Z[x_e, x_{C_2}, x_{C_4}]$, the $C_2$-factor of $\Barks(C_4)\otimes_{\Marks(C_4)}\Barks(C_4)$ is $\Z[x_e, x_{C_2}] \otimes \Z[x_e, x_{C_2}]$. The $C_2$-coordinate of the comultiplication $((\chi\otimes \chi)(\Delta^\times s))_{C_2}$ can therefore be written in the form
\begin{equation*}
    \sum_{i,j,k,l} h_{i,j,k,l} \left(x_e^{i} x_{C_2}^{j} \otimes x_e^{k} x_{C_2}^{l}\right),
\end{equation*}
with $h_{i,j,k,l}\in \Z$ for all $i,j,k,l \in \N$.
Evaluating $(\chi\otimes\chi)(\Delta^\times s)$ at $(\bar{u},\bar{v})$ with $\bar{u} = (u_1,u_2,u_3)$ and $\bar{v} = (v_1,v_2,v_3)$ then yields an element whose $C_2$-coordinate is
\begin{equation*}
    \sum_{i,j,k,l} h_{i,j,k,l} \left(u_1^{i} u_{2}^{j} v_1^{k} v_2^{l}\right).
\end{equation*}
Equating this with the $C_2$-coordinate of the right-hand side, we get
\begin{equation}
    \label{zero_polynomial_on_lattice}
    2u_1v_1 - \sum_{i,j,k,l} h_{i,j,k,l} \left(u_1^{i} u_{2}^{j} v_1^{k} v_2^{l}\right) = 0    
\end{equation}
for all $(u_1, u_2, u_3), (v_1, v_2, v_3) \in \chi(P^+(C_4))\subseteq \Barks(C_4)$. In particular, we note that since 
\begin{align*}
    \chi(u) = (4,0,0), \ \chi(v) = (2,2,0), \text{ and } \chi(1) = (1,1,1),
\end{align*}
for $u,v,1 \in A(C_4) \subseteq P(C_4)$, the image of the effective Burnside--Tambara semiring under $\chi$ contains all tuples of the form $(4l_1, 4l_2, 4l_3)$ where $l_1\geq l_2\geq l_3 \in \mathbb{Z}$. So \cref{zero_polynomial_on_lattice} holds for all choices of $u_1,u_2,v_1,v_2\in 4\N$ with $u_1 \geq u_2$ and $v_1 \geq v_2$, i.e. the polynomial 
\begin{equation*}
    2ac - \sum_{i,j,k,l} h_{i,j,k,l} \left(a^ib^jc^kd^l\right) \in \Z[a,b,c,d]
\end{equation*}
evaluates to $0$ on all such $u_1,u_2,v_1,v_2$ and hence must be the zero polynomial. Therefore, the $C_2$-coordinate of $(\chi\otimes\chi)(\Delta^\times s)$ is
\begin{equation*}
    2x_e \otimes x_e \in \Z[x_e, x_{C_2}] \otimes \Z[x_e, x_{C_2}].
\end{equation*}
Let
    $$\pi_2 \colon \Barks(C_4)\otimes_{\Marks(C_4)} \Barks(C_4) \to \Z[x_e, x_{C_2}] \otimes \Z[x_e, x_{C_2}]$$
be the projection onto the $C_2$-coordinate and let
    $$f\colon \Z[x_e, x_{C_2}] \otimes \Z[x_e, x_{C_2}] \to \Z[y, z]$$
be the map given by
\begin{align*}
    f(x_e \otimes 1) &= y \\
    f(1 \otimes x_e) &= z \\
    f(x_{C_2} \otimes 1) &= 0 \\
    f(1 \otimes x_{C_2}) &= 0.
\end{align*}
Then we have, from the above computation, that $(f\circ \pi_2 \circ (\chi \otimes \chi))(\Delta^\times s) = 2yz$. However, this is a contradiction because $2yz$ is not in the image of $f \circ \pi_2 \circ (\chi \otimes \chi)$: we know that $u,v,x,m,n,s$ generate $P(C_4)$ as a ring (\cref{big_burnside_ring_C4_example_computation}), which gives us a collection of $10$ generators for $P(C_4) \otimes_{A(C_4)} P(C_4)$ (since $u, v\in A(C_4)$, $1\otimes u = u \otimes 1$ and $1\otimes v = v \otimes 1$). Applying $f \circ \pi_2 \circ (\chi \otimes \chi)$ to these $10$ generators, the only non-constant elements we get are $y^2, z^2, 2y, 2z$. Therefore, the image of this map is the subring $\Z[y^2, z^2, 2y, 2z] \subset \Z[y, z]$, and this subring does not contain $2yz$.
\end{proof}

However, we do have a positive result for the rationalization of $P(G)$.
\begin{proposition}
    If $G$ is Dedekind, then $\Q \otimes P(G)$ admits the structure of a $(\Q\otimes A(G))$-plethory compatibly with the composition of bispans.
\end{proposition}
\begin{proof}
    By \cref{classicalmarks,character_map_rational_iso,compatibility_with_composition_in_barks}, this is equivalent to the statement that $\Q\otimes\Barks(G)$ is a $(\Q\otimes \Marks(G))$-plethory, which we noted in \cref{relative_barks}.
\end{proof}

So although we cannot define cooperations $\Delta^+, \Delta^\times \colon P(G) \to P(G)\otimes_{A(G)} P(G)$ generally, we do have the restricted rational cooperations
$$\Delta^+, \Delta^\times \colon P(G) \to \Q\otimes \left(P(G) \otimes_{A(G)} P(G)\right)$$
which can be used to compute composites just the same. For example, for $s\in P(C_4)$ as in \cref{C4notplethory}, we have $\Delta^\times(s) = \frac{1}{2} s\otimes s$, and so for any $a, b\in P(C_4)$, we have
$$s\circ (ab) = \left(\frac{1}{2} s\otimes s\right)(a,b) = \frac{1}{2}(s\circ a)(s\circ b).$$
This is an element of $P(C_4)$ because $(s\circ a)(s\circ b)$ is always divisible by $2$.

\bibliographystyle{alpha}
\bibliography{references.bib}

\end{document}